\documentclass[a4paper, 10pt, journal]{ieeeconf}      %
\IEEEoverridecommandlockouts %
\renewcommand{\QED}{\QEDopen}

\usepackage{amsmath} %
\usepackage{thmtools}

\newtheorem{theorem}{Theorem}[section]
\newtheorem{proposition}[theorem]{Proposition}

\newtheorem{lemma}[theorem]{Lemma}
\newtheorem{corollary}[theorem]{Corollary}

\newtheorem{remark}[theorem]{Remark}

\newtheorem{open problem}[theorem]{Open problem}
\newtheorem{definition}[theorem]{Definition}

\newenvironment{proofof}[1]{\par\mbox{\textit{\it Proof of #1:} ~}\ignorespaces}{\strut\hfill\QED}

\usepackage{amsfonts} %
\usepackage{amssymb} %
\usepackage{amsopn} %
\usepackage{ifthen}
\usepackage{suffix} %

\newcommand \mc {\mathcal}

\DeclareMathAlphabet{\mbb}{U}{bbold}{m}{n} 
\newcommand \mb [1] {\ifthenelse{\equal{#1}{0}}{\mbb{0}}{\ifthenelse{\equal{#1}{1}}{\mbb{1}}{\mathbb{#1}}}} %

\newcommand \RR {\mb R}

\newcommand{\on}{\operatorname}

\newcommand \inv {^{-1}}
\newcommand \T {^\top}
\newcommand \Tinv {^{-\top}}
\newcommand \invT {^{-\top}}
\renewcommand \H {^\mathrm{*}}
\newcommand \comp {^{\mathrm{c}}}

\WithSuffix \newcommand \inv* {\mathstrut^{-1}}
\WithSuffix \newcommand \T* {\mathstrut^\top}
\WithSuffix \newcommand \Tinv* {\mathstrut^{-\top}}
\WithSuffix \newcommand \invT* {\mathstrut^{-\top}}
\WithSuffix \newcommand \H* {\mathstrut^\mathrm{*}}
\WithSuffix \newcommand \comp* {\mathstrut^{\mathrm{c}}}

\let\tilde\widetilde

\let\hat\widehat

\renewcommand \subset {\subseteq}

\newcommand \diag [1] {\on{diag}}
\newcommand {\cl} [1] {\operatorname{cl}(#1)}
\newcommand {\conv} [1] {\operatorname{conv}(#1)}
\newcommand {\inter} [1] {\operatorname{int}(#1)}

\newcommand {\myvphantom} [1] {\let \\ \empty \vphantom{#1}}
\newcommand {\set} [2] {%
\left\{%
\myvphantom{#1#2}%
\right.%
#1%
\left|%
\myvphantom{#1#2}%
\right.#2%
\left.%
\myvphantom{#1#2}%
\right\}%
}

\newcommand{\pdd}[2]{\frac{\partial#1}{\partial#2}}

\newcommand \Ie {\textit{I.e.}}
\def \ie {\textit{i.e.}}
\newcommand \eg {\textit{e.g.}}

\renewcommand{\l}{\left}
\renewcommand{\r}{\right}
\newcommand{\nolabel}{\nonumber}

\usepackage{pdfpages} %
\usepackage{xcolor} %
\usepackage{graphicx} %
\usepackage{enumerate} %

\usepackage{xr} %
\newcommand{\PIref}[1]{\ref{P1-#1} of Part~I}

\def\mytitle{DC power grids with constant-power loads\textemdash Part II: nonnegative power demands, conditions for feasibility, and high-voltage solutions}
\def\mythanks{This work is supported by NWO (Netherlands Organisation for Scientific Research) project `Energy management strategies for interconnected smart microgrids' within the DST-NWO Joint Research Program on Smart Grids.}

\def\myauthors{Mark Jeeninga, Claudio De Persis and Arjan van der Schaft}
\def\myaffiliation{University of Groningen, 9747AG Groningen, The Netherlands (e-mail: \{m.jeeninga, c.de.persis, a.j.van.der.schaft\}@rug.nl)}

\title{\mytitle}

\author{\myauthors%
	\thanks{\mythanks}%
	\thanks{\myaffiliation}%
}

\begin{document}

\maketitle
\thispagestyle{empty}
\pagestyle{empty}

\begin{abstract} 
In this two-part paper we develop a unifying framework for the analysis of the feasibility of the power flow equations for DC power grids with constant-power loads.

Part II of this paper explores further implications of the results in Part~I. 
In particular, we refine several results in Part~I to obtain a necessary and sufficient condition for the feasibility of nonnegative power demands, which is cheaper to compute than the necessary and sufficient LMI condition in Part~I. 
Moreover, we prove two novel sufficient conditions, which generalize known sufficient conditions for power flow feasibility in the literature.
In addition, we prove that the unique long-term voltage semi-stable operating point associated to a feasible vector of power demands is a strict high-voltage solution.
A parametrization of such operating points, which is dual to the parametrization in Part~I, is also obtained, along with a parametrization of the boundary of the set of feasible power demands.
\end{abstract}
\section{Introduction}
The feasibility of the power flow equations is of crucial importance for the long-term safe operation of a power grid.
Classical papers such as \cite{lof1993analysis,62415,tinney1967power} have studied this problem for AC power grids, and over the past decade, the research for AC power grids has been reinvigorated by articles such as \cite{dorfler2013novel,dymarsky2014convexity,bolognani2015existence,simpson2016voltage,barabanov2016}.
Unfortunately, a complete understanding of this problem is still lacking.

Similar to the AC case, the somewhat simpler case concerning DC power grids is also not well-understood.
A notable advancement is \cite{matveev2020tool}, which presents an algorithm to decide on the feasibility of the DC power flow equations with constant-power loads. %
However, a full characterization of the feasibility of the DC power flow equations is not found in the literature.
For a more detailed introduction we refer to Part~I of this paper.

The aim of this twin paper is to provide an in-depth analysis of the power flow equations of DC power grids with constant-power loads, and develop a framework which unifies and extends known results in the literature. 
In Part~I we presented a complete geometric characterization of the feasibility of the associated power flow equations. More importantly, we obtained necessary and sufficient conditions for their feasibility, and presented a method to compute the corresponding long-term voltage semi-stable operating point, which was shown to be unique.
These advances fill an important gap in the literature, and provide a deep insight in the nature of power flow feasibility and voltage stability of power grids with constant-power loads. In Part~II of this paper continues this approach by studing nonnegative power demands, sufficient conditions for feasibility, and high-voltage solutions.

\subsection*{Contribution}
The main objective of this twin paper is to analyze the set, denoted as $\mc F$, of constant power demands for which the power flow equations are feasible, and their associated operating points.
In the following we let $\mc D$ denote the set of long-term voltage stable operating points. 
We refer to the vectors in $\cl{D}$, the closure of $\mc D$, as \emph{long-term voltage semi-stable} operating points.
The main contributions of Part~I of this paper are summarized as follows.
\newcommand{\mainresult}[1]{\textbf{M\ref{main results:#1}}}
\begin{enumerate}[\bf M1.]
\item We give a parametrization of $\mc D$, its closure and its boundary, which establishes a constructive method to describe the long-term voltage (semi-)stable operating points. \label{main results:parametrization of D}
\item For each vector of power demand that lies on the boundary of $\mc F$ there exists a unique corresponding operating point which solves the power flow equations. Moveover, these operating points form the boundary of the set $\mc D$.%
\label{main results:one-to-one boundary}%
\item There is a one-to-one correspondence between the feasible power demands $\mc F$ and the long-term voltage semi-stable operating points $\cl{\mc D}$. This means that if the power flow equations are feasible, then there exists a unique long-term voltage semi-stable operating point that solves the power flow equations.
This operating point can be found by solving an initial value problem.\label{main results:one-to-one correspondence}%
\item We give a novel and insightful proof for the fact that the set $\mc F$ is closed and convex. Consequently, $\mc F$ is the intersection of all supporting half-spaces of $\mc F$. We describe all such half-spaces, which gives a complete geometric characterization of $\mc F$.\label{main results:convexity of F}%
\item %
We prove a necessary and sufficient LMI condition for the feasibility of the power flow equations, %
and a necessary and sufficient LMI condition for the feasibility of the power flow equations under small perturbations.\label{main results:necessary and sufficient condition}%
\end{enumerate}
The aim of Part~II %
is to extend and unify known results in the literature by elaborating on the framework of Part~I.
In particular, we study the case where all power demands are nonnegative, generalize sufficient conditions in the literature, and look at the operating points in more detail.
The main results of Part~II are as follows.
\begin{enumerate}[\hspace{5pt}\bf M1.]\setcounter{enumi}{5}
\item We give an alternative parametrization of $\mc D$, its closure and its boundary (Theorem~\ref{theorem:alternative parametrization of D}), which is in a sense dual to the parametrization mentioned in \mainresult{parametrization of D}. \label{main results:alternative parametrization of D}
\item We give two parametrizations of $\partial\mc F$, the boundary of the set of feasible power demands (Theorem~\ref{theorem:parametrization of boundary of F}, Corollary~\ref{corollary:alternative boundary of F}). %
\label{main results:boundary of F}%
\item We consider the restriction of $\mc F$ to nonnegative power demands, and present a parametrization for $\partial\mc F$ for nonnegative power demands (\textbf{a}) which is cheaper to compute than \mainresult{boundary of F}. Moreover, we deduce a refinement of the necessary and sufficient condition \mainresult{necessary and sufficient condition} for nonnegative power demands (\textbf{b}) which is cheaper to compute than \mainresult{necessary and sufficient condition} (Theorem~\ref{theorem:parametrization nonnegative boundary of F}, Theorem~\ref{theorem:nonnegative necessary and sufficient condition}).\label{main results:nonnegative power demands}
\item We prove that any vector of power demands that is element-wise dominated by a feasible vector of power demands is also feasible (Lemma~\ref{lemma:power demand domination}).\label{main results:power demand domination}
\item We present two novel sufficient conditions for the feasibility of the power flow equations which generalize the sufficient conditions in \cite{simpson2016voltage} and \cite{bolognani2015existence} (Corollary~\ref{corollary:sufficient condition nonnegative}, Theorem~\ref{theorem:generalization simpson-porco condition}), and show how these conditions are related (Lemma~\ref{lemma:bolognani boundary}).\label{main results:sufficient conditions}
\item We show that the long-term voltage stable operating point is a strict high-voltage solution (Theorem~\ref{theorem:dominating solution}).
Consequently, the operating points associated to a feasible power demand which are either long-term voltage stable, a high-voltage solution, or dissipation-minimizing, are one and the same (Theorem~\ref{theorem:desirable operating point}).\label{main results:desirable operating point}
\end{enumerate}
We briefly discuss how these results are related to the literature.
Regarding \mainresult{convexity of F}, the convexity of $\mc F$ was shown in \cite{dymarsky2014convexity}.
The result \mainresult{necessary and sufficient condition} proves that the necessary condition in \cite{barabanov2016} is also sufficient, and further extends the condition to power demands which are feasible under small perturbations.
The results of \mainresult{sufficient conditions} generalize the sufficient conditions of \cite{bolognani2015existence} and \cite{simpson2016voltage}.
The paper \cite{matveev2020tool} proved that, if the power flow equations are feasible, then there exists a high-voltage solution which is ``almost surely'' long-term voltage stable. In addition, \cite{matveev2020tool} gives a sufficient condition for which this is the unique long-term voltage stable operating point associated to a vector of power demands.
We show that this operating point is a strict high-voltage solution and always %
coincides with the unique long-term voltage semi-stable operating point (\mainresult{desirable operating point}, \mainresult{one-to-one correspondence}).\\
We refer the reader to Part~I for a more detailed discussion. %

\subsection*{Organization of Part~II}

In Section~\ref{section:summary of part 1} a summary is given of the models, definitions and results from Part~I of this paper. %

Section~\ref{section:nonnegative power demands} focuses on nonnegative power demands, and studies when such power demands are feasible. 
First, we give an alternative parametrization of $\mc D$ and discuss its relation to the parametrization of $\mc D$ in Part~I (\mainresult{alternative parametrization of D}). By means of this parametrization we study the boundary of $\mc F$ (\mainresult{boundary of F}), and derive a parametrization for the boundary of feasible power demands in the nonnegative orthant (\mainresult{nonnegative power demands}\textbf{a}).
This allows us to refine the necessary and sufficient condition \mainresult{necessary and sufficient condition} for nonnegative power demands (\mainresult{nonnegative power demands}\textbf{b}). %

Section~\ref{section:sufficient conditions} recovers and generalizes several sufficient conditions in the literature in the context of DC power grids.
More specifically, we prove two sufficient conditions (\mainresult{sufficient conditions}) which generalize the sufficient conditions in \cite{simpson2016voltage} and %
\cite{barabanov2016}.
In addition, we show that any power demand which is element-wise dominated by a feasible power demand is feasible as well (\mainresult{power demand domination}).

Section~\ref{section:high-voltage solution} focuses on the long-term voltage semi-stable operating points.
We show that any such operating point is a strict high-voltage solution. As a consequence, the notions of long-term voltage stable operation points, dissipation-minimizing operation points and (strict) high-voltage solutions coincide (\mainresult{desirable operating point}). %

Section~\ref{section:conclusion} concludes the paper.

\subsection*{Notation and matrix definitions}
For a vector $x = \begin{pmatrix}
x_1 & \cdots & x_k
\end{pmatrix}\T$ we denote 
\begin{align*}
[x]:=\on{diag} (x_1,\dots,x_k). 
\end{align*}
We let $\mb 1$ and $\mb 0$ denote the all-ones and all-zeros vector, respectively, and let $I$ denote the identity matrix.
We let their dimensions follow from their context.
All vector and matrix inequalities are taken to be element-wise. 
We write $x \lneqq y$ if $x\le y$ and $x\neq y$.
We let $\|x\|_p$ denote the $p$-norm of $x\in\RR^k$.

We define $\boldsymbol n := \{1,\dots,n\}$.
All matrices are square $n\times n$ matrices, unless stated otherwise.
The submatrix of a matrix $A$ with rows and columns indexed by $\alpha,\beta\subset\boldsymbol n$, respectively, is denoted by $A_{[\alpha,\beta]}$. The same notation $v_{[\alpha]}$ is used for subvectors of a vector $v$.
We let $\alpha\comp$ denote the set-theoretic complement of $\alpha$ with respect to $\boldsymbol n$. 
For a set $S$, the notation $\inter S$, $\cl S$, $\partial S$ and $\conv S$ is used for the interior, closure, boundary and convex hull of $S$, respectively.

We list some classical definitions from matrix theory.
\begin{definition}[\cite{fiedler1986special}, Ch. 5]\label{definition:Z-matrix}
A matrix $A$ is a \emph{Z-matrix} if $A_{ij}\le 0$ for all $i\neq j$.
\end{definition}
\begin{definition}[\cite{fiedler1986special}, Thm. 5.3]\label{definition:M-matrix}
A Z-matrix is an \emph{M-matrix} if all its eigenvalues have nonnegative real part.
\end{definition}
\begin{definition}[\cite{fiedler1986special}, pp. 71]\label{definition:irreducible matrix}
A matrix $A$ is \emph{irreducible} if for every nonempty set $\alpha\subsetneqq \boldsymbol n$ we have $A_{[\alpha,\alpha\comp]}\neq 0$.
\end{definition}

\begin{definition}
The Schur complement of $M = \begin{pmatrix}
A & B \\ C& D 
\end{pmatrix}$ with respect to the principal submatrix $D$ is denoted by 
\begin{align*}
M/D := A - BD\inv C.
\end{align*}
\end{definition}

\section{Summary of Part I}\label{section:summary of part 1}
We summarize the models, definitions and main contributions of Part I of this paper. %

\subsection{The DC power grid model}
Throughout this paper we consider DC power grids with constant-power loads at steady state, and subsequently study their power flow equations.
Since dynamic components do not contribute to the power flow equations at steady state, we model such power grids as resistive circuits. We refer to \cite{schaft2010characterization,schaft2019flow} for a detailed discussion on resistive circuits.

For a given power grid with $n$ loads and $m$ sources, we let $Y\in\RR^{(n+m)\times (n+m)}$ denote its Kirchhoff matrix, mapping the voltage potentials $V$ at the nodes to the currents $\mc I$ injected at the nodes. It is partitioned as
\begin{align}
Y = \begin{pmatrix}\label{eqn:partition of Y}
Y_{LL} & Y_{LS} \\ Y_{SL} & Y_{SS}
\end{pmatrix}
\end{align}
according to whether nodes are loads ($L$) and sources ($S$).
Following the same partition we let
\begin{align*}
V = \begin{pmatrix}
V_L \\ V_S
\end{pmatrix}\in\RR^{n + m};\quad \mc I = \begin{pmatrix}
\mc I_L \\ \mc I_S
\end{pmatrix}\in\RR^{n + m}
\end{align*}
denote the voltage potentials and the currents at the loads and sources, respectively.
All voltage potentials are assumed to be positive (\ie, $V>\mb 0$).
The power $P\in\RR^{n + m}$ which the nodes supply to the grid is given by
\begin{align*}%
\begin{pmatrix}
P_L \\ P_S
\end{pmatrix} = \begin{bmatrix}
\begin{pmatrix}
V_L \\ V_S
\end{pmatrix}
\end{bmatrix}\begin{pmatrix}
\mc I_L \\ \mc I_S
\end{pmatrix} = \begin{bmatrix}
\begin{pmatrix}
V_L \\ V_S
\end{pmatrix}
\end{bmatrix}\begin{pmatrix}
Y_{LL} & Y_{LS} \\ Y_{SL} & Y_{SS}
\end{pmatrix}\begin{pmatrix}
V_L \\ V_S
\end{pmatrix}
\end{align*}
as follows from Kichhoff's and Ohm's laws.
The \emph{total dissipated power} in the lines equals 
\begin{align*}
R(V_L,V_S) := V\T Y V \ge 0.
\end{align*}

\subsection{Feasible power demands}
Throughout this paper we consider $V_L$ as a variable of the system, whereas $Y$ and $V_S>\mb 0$ are known (fixed) para-meters. 
We therefore write $\mc I_L = \mc I_L(V_L)$ and $P_L = P_L(V_L)$.
\begin{definition}
We define the \emph{source-injected currents} by 
\begin{align}\label{eqn:source-injected currents}
\mc I_L^*:=-Y_{LS}V_S = -\mc I_L(\mb 0),
\end{align}
which correspond to the currents injected into the loads when $V_L=\mb 0$.
\end{definition}
\begin{proposition}[{\cite[Prop. 3.6]{schaft2010characterization}}]
The \emph{open-circuit voltages} $V_L^*$ are the unique voltage potentials at the loads so that $\mc I_L(V_L) = \mb 0$, given by
\begin{align}\label{eqn:open-circuit voltages}
V_L^* := -Y_{LL}\inv* Y_{LS}V_S = Y_{LL}\inv* \mc I_L^*.
\end{align}
\end{proposition}
\begin{lemma}[Lemma~\PIref{lemma:positive open-circuit voltages}]\label{lemma:positive open-circuit voltages}
The source-injected currents $\mc I_L^*$ are nonnegative and not all zero (\ie, $\mc I_L^*\gneqq \mb 0$) and the open-circuit voltages $V_L^*$ are positive (\ie, $V_L^*>\mb 0$).
\end{lemma}

The power injected at the loads for every $V_L$ is given by %
\begin{align}\label{eqn:power at the nodes}
P_L(V_L) = [V_L]Y_{LL} (V_L - V_L^*).
\end{align}
The vector $V_L^*$ is the unique positive vector such that $P_L(V_L) = \mb 0$. %

In this paper we consider \emph{constant-power loads}, which is to say that each load demands a fixed quantity of power from the power grid.
We let $P_c\in\RR^n$ denote the vector of constant power demands at the loads.
Note that we do not impose any sign restrictions on $P_c$ and that, in principle, load nodes could also demand negative power, in which case the loads provide constant power to the grid.
The \emph{DC power flow equations for constant-power loads} are given by
\begin{align}\label{eqn:dc power flow equation}
[V_L]Y_{LL} (V_L - V_L^*) + P_c = \mb 0.
\end{align}

\begin{definition}\label{definition:feasible}
Given $Y$ and $V_S$, we say that the power flow equations \eqref{eqn:dc power flow equation} are \emph{feasible} for a vector of constant power demands $P_c$ if there exists a vector of voltage potentials $V_L$ %
which satisfies \eqref{eqn:dc power flow equation}. We say that $V_L$ is an \emph{operating point} associated to $P_c$ if $V_L$ satisfies \eqref{eqn:dc power flow equation} for $P_c$. 
\end{definition}

Recall that throughout Definition~\ref{definition:feasible} we require that ${V_S>\mb 0}$ and $V_L>\mb 0$. %
\begin{definition}
We say that a vector of power demands $P_c$ is \emph{feasible} if \eqref{eqn:dc power flow equation} is feasible for $P_c$.
The set of \emph{feasible power demands} is given by
\begin{align*}
\mc F :=& \set{P_c\in\RR^n}{\text{Eq. \eqref{eqn:dc power flow equation} is feasible for $P_c$}}\\
=&\set{P_c\in\RR^n}{\exists V_L>\mb 0 \text{ such that \eqref{eqn:dc power flow equation} holds}}.%
\end{align*}
\end{definition}

Note that each vector $V_L$ of voltage potentials at the loads is associated by \eqref{eqn:dc power flow equation} to a vector of constant power demands $P_c$, %
given by
\begin{align}\label{eqn:definition of f}
P_c(V_L) = [V_L]Y_{LL}(V_L^*-V_L).
\end{align}
Of particular interest is the feasible power demand which maximizes the total power demanded by all loads in the power grid.
\begin{definition}
For a feasible power demand $P_c\in \mc F$, the \emph{total feasible power demand} is $\mb 1\T P_c$, the sum of the power demands at the loads.
\end{definition}
\begin{definition}\label{definition:maximizing feasible power demand}
A \emph{maximizing feasible power demand} is a feasible power demand $P_{\text{max}}\in \mc F$ which maximizes the total feasible power demand. Thus for all $P_c\in \mc F$ it satisfies
\begin{align}\label{eqn:power flow necessary condition}
\mb 1\T P_c \le \mb 1\T P_{\text{max}}.
\end{align}
\end{definition}
\begin{lemma}[Lemma~\PIref{lemma:maximizing power demand}]\label{lemma:maximizing power demand}
There is a unique maximizing feasible power demand $P_{\text{max}}\in \mc F$, given by
\begin{align}\label{eqn:max total power demand}
P_{\text{max}} = \tfrac 1 4 [V_L^*]\mc I_L^* \gneqq \mb 0.
\end{align}
The unique operating point corresponding to $P_{\text{max}}$ is $\tfrac 1 2 V_L^*$.
\end{lemma}

The following definitions are used to characterize the set $\mc F$ of feasible power demands.
We define the notation 
\begin{align}\label{eqn:definition of h}
h(\lambda):=\tfrac 1 2 ([\lambda]Y_{LL} + Y_{LL}[\lambda]),
\end{align} along with the sets 
\begin{align*}
\Lambda &:= \set{\lambda\in\RR^n}{h(\lambda)\text{ is positive definite}};\\
\Lambda_1 &:= \set{\lambda\in\RR^n}{h(\lambda)\text{ is positive definite}, \|\lambda\|_1 = 1}.
\end{align*}
For each $\lambda\in \Lambda$ we define the norm, map, and half-space
\begin{align}
\|x\|_{h(\lambda)} &:= \sqrt{x\T h(\lambda) x}\label{eqn:definition of norm};\\
\varphi (\lambda) &:= \tfrac 1 2 h(\lambda)\inv [\lambda]\mc I_L^*\label{eqn:definition of phi};\\
H_{\lambda} &:= \set{y}{\lambda\T y \le \|\varphi(\lambda)\|_{h(\lambda)}^2},\label{eqn:definition of H}
\end{align}
respectively. 
We note that $h(\mb 1) = Y_{LL}$, $\varphi(\mb 1) = \tfrac 1 2 V_L^*$, $P_c(\varphi(\mb 1)) = P_{\mathrm{max}}$, and that \eqref{eqn:power flow necessary condition} is equivalent to the inclusion 
\begin{align}\label{eqn:Pmax half-space}
\mc F \subset H_{\mb 1} = \set{y}{\mb 1 \T y \le \mb 1 \T P_{\mathrm{max}}}.
\end{align}

\begin{theorem}[\mainresult{convexity of F} \textendash\ Theorem~\PIref{theorem:convexity of F}]\label{theorem:convexity of F}
The set $\mc F$ is closed, convex and is the intersection over all $\lambda\in\Lambda_1$ of the half-spaces $H_\lambda$. \Ie
\begin{align*}
\mc F = \cl{\conv{\mc F}} =  \bigcap_{\lambda\in\Lambda_1} H_\lambda.
\end{align*}
\end{theorem}
In Part~I it was shown that the necessary condition in \cite{barabanov2016} for feasibility of a power demand is also sufficient, and can be sharpened for power demands which are feasible under small perturbation, by which we mean that $\tilde P_c$ is feasible and does not lie on the boundary of $\mc F$ (\ie, $\tilde P_c\in\inter{\mc F}$).
\begin{theorem}[\mainresult{necessary and sufficient condition} \textendash\ Theorem~\PIref{theorem:necessary and sufficient condition}]\label{theorem:necessary and sufficient condition}
A vector $\tilde P_c$ of power demands is feasible if and only if there does not exist a positive vector $\nu\in\RR^n$ such that the $(n+1) \times (n+1)$ matrix
\begin{align}\label{eqn:necessary and sufficient condition}
\begin{pmatrix}
[\nu] Y_{LL} + Y_{LL}[\nu] & [\nu] \mc I_L^* \\ ([\nu] \mc I_L^*)\T & 2 \nu\T \tilde P_c
\end{pmatrix} = 2 \begin{pmatrix}
h(\nu) & \tfrac 1 2 [\nu] \mc I_L^* \\ \tfrac 1 2 ([\nu] \mc I_L^*)\T & \nu\T \tilde P_c
\end{pmatrix}
\end{align}
is positive definite.
Similarly, $\tilde P_c$ is feasible under small perturbation (\ie, $\tilde P_c\in\inter{\mc F}$) if and only if there does not exist a positive vector $\nu\in\RR^n$ such that \eqref{eqn:necessary and sufficient condition} is positive semi-definite.
\end{theorem}

\subsection{Desirable operating points}
For a given vector of power demands $P_c$ there may exist multiple operating points $V_L$ which satisfy \eqref{eqn:dc power flow equation}.
In this paper we consider two types of operating points which are in some sense desirable.

\paragraph{Long-term voltage stable operating points}
The Jacobian of $P_c(V_L)$ is given by
\begin{align}\label{eqn:jacobian of P_c}
\pdd {P_c}{V_L}(V_L) = [Y_{LL}(V_L^* - V_L)] - [V_L]Y_{LL}.
\end{align}
\begin{definition}\label{definition:long-term voltage stable}
An operating point $\tilde V_L$ associated to $\tilde P_c$ is \emph{long-term voltage stable} if the Jacobian of $P_c(V_L)$ at $\tilde V_L$ is nonsingular, %
and its inverse is a matrix with negative elements (\ie,\footnote{The equality $\pdd {P_c}{V_L}(\tilde V_L)\inv = \pdd {V_L}{P_c}(\tilde P_c)$ follows from the Inverse Function Theorem, see \eg\ \cite{rudin1964principles}.} $\pdd {P_c}{V_L}(\tilde V_L)\inv = \pdd {V_L}{P_c}(\tilde P_c) < 0$). The set of all long-term voltage stable operating points is defined by
\begin{align*}
\mc D := \set{V_L}{%
\begin{minipage}{162pt}\text{$\exists P_c$ such that $V_L$ is a long-term voltage} \\\text{stable operating point associated to $P_c$}\end{minipage}
}.
\end{align*}
\end{definition}\vspace{4pt}
We remark that there are many equivalent definitions of long-term voltage stability DC power grids with constant power loads (see Remark~\PIref{remark:equivalent definitions of long-term voltage stability}).
Similar to Definition~\ref{definition:long-term voltage stable} we define the notion of long-term voltage semi-stability:
\begin{definition}\label{definition:long-term voltage semi-stable}
An operating point $\tilde V_L$ associated to $\tilde P_c$ is \emph{long-term voltage semi-stable} if for every $\varepsilon>0$ there exists a long-term voltage stable operating point $\widehat V_L$ associated to some $\widehat P_c$ such that $\|\tilde V_L- \widehat V_L\|_2<\varepsilon$. Consequently, the set of all long-term voltage semi-stable operating points equals $\cl{\mc D}$, the closure of $\mc D$.
\end{definition}
\begin{lemma}[Prop.~\ref{P1-proposition:jacobian M-matrix} and Cor.~\ref{P1-corollary:boundary of D} of Part I]\label{lemma:D m-matrix}
The set $\mc D$ of long-term voltage stable operating points,
and its closure $\cl{\mc D}$ of long-term voltage semi-stable operating points, and its boundary $\partial \mc D$ satisfy
\begin{align*}
&\mc D = \set{V_L>\mb 0}{ -\pdd {P_c}{V_L}(V_L) \text{ is a nonsingular M-matrix}};\\
&\cl{\mc D} = \set{V_L>\mb 0}{ -\pdd {P_c}{V_L}(V_L) \text{ is an M-matrix}};\\
&\partial \mc D = \set{V_L>\mb 0}{ -\pdd {P_c}{V_L}(V_L) \text{ is a singular M-matrix}}.
\end{align*}
\end{lemma}

\begin{theorem}[\mainresult{parametrization of D}\textemdash Theorem~3.12 of Part I]\label{theorem:parametrization of D}
The set $\mc D$ of long-term voltage stable operating points, its closure $\cl{\mc D}$, and its boundary $\partial\mc D$ are parametrized by
\begin{align*}
\mc D &= \set{\varphi(\lambda) + r h(\lambda)\inv \lambda }{\lambda\in\Lambda_1, r>0};\\
\cl{\mc D} &= \set{\varphi(\lambda) + r h(\lambda)\inv \lambda }{\lambda\in\Lambda_1, r\ge 0};\\
\partial\mc D &= \set{\varphi(\lambda)}{\lambda\in\Lambda_1}.
\end{align*}
\end{theorem}

Part~I shows that the boundary of $\mc F$ and the boundary of $\mc D$ are in one-to-one correspondence, and that this result extends to a one-to-one correspondence between $\mc F$ and $\cl{\mc D}$.
\begin{theorem}[\mainresult{one-to-one boundary}\textemdash Corollary~\PIref{corollary:one-to-one boundary}]\label{theorem:one-to-one boundary}
For each $\tilde P_c$ on the boundary of $\mc F$ there exist a unique $\tilde V_L\in\RR^n$ %
that satisfies \eqref{eqn:dc power flow equation}. All such $\tilde V_L$ satisfy $\tilde V_L>\mb 0$ and form the boundary of $\mc D$. 
This implies that there is a one-to-one correspondence between $\partial \mc D$ and $\partial \mc F$.
\end{theorem}

In Part~I it was shown that Theorem~\ref{theorem:one-to-one boundary} extends to a one-to-one correspondence between $\mc F$ and $\cl{\mc D}$.
This correspondence was made explicit by showing that the operating point $\tilde V_L$ in $\cl{\mc D}$ corresponding to a feasible power demand $\tilde P_c$ in $\mc F$ can be found by solving an initial value problem.
\begin{theorem}[\mainresult{one-to-one correspondence} \textemdash Theorem~\PIref{theorem:one-to-one correspondence}]\label{theorem:one-to-one correspondence}
There is a one-to-one correspondence between the long-term voltage semi-stable operating points $\cl{\mc D}$ and the feasible power demands. \Ie, for each $\tilde P_c\in\mc F$ there exists a unique $\tilde V_L\in\cl{\mc D}$ which satisfies $\tilde P_c = P_c(\tilde V_L)$, implying that $\mc F = P_c(\cl{\mc D})$. More explicitly, $\tilde V_L$ is obtained by solving the initial value problem
\begin{align}\label{eqn:one-to-one correspondence:initial value problem}
\dot \gamma(\theta) = \l(\pdd {P_c} {V_L} (\gamma(\theta))\r)\inv \tilde P_c %
\end{align}
for $\gamma: [0,1]\to \RR^n$ with initial value $\gamma(0)=V_L^*$, where the solution $\gamma$ exists, is unique, and satisfies $\gamma(1)=\tilde V_L$.
\end{theorem}

Theorem~\ref{theorem:one-to-one boundary} and Theorem~\ref{theorem:one-to-one correspondence} together imply the following corollary.
\begin{corollary}\label{corollary:interior F one-to-one D}
There is a one-to-one correspondence between the long-term voltage stable operating points $\mc D$ and the feasible power demands under small perturbations $\inter{\mc F}$.
\end{corollary}

\paragraph{Dissipation-minimizing operating points}
It is desirable that an operating point which satisfies \eqref{eqn:dc power flow equation} minimizes $R(V_L,V_S)$, the total power dissipated in the lines. 
\begin{definition}\label{definition:dissipation-minimizing operating point}
Given $P_c$, an operating point $\tilde V_L$ associated to $\tilde P_c$ is \emph{dissipation-minimizing} if for all operating points $V_L$ associated to $\tilde P_c$ we have $R(\tilde V_L,V_S)\le R(V_L,V_S)$.
\end{definition}
Related to the dissipation-minimizing operating points are the high-voltage solutions.
\begin{definition}\label{definition:high-voltage solution}
An operating point $\tilde V_L$ associated to the power demands $P_c$ is a \emph{high-voltage solution} if $\tilde V_L \ge V_L$ for all $V_L$ associated to $P_c$, and is a \emph{strict high-voltage solution} if $\tilde V_L > V_L$ for all $V_L\neq\tilde V_L$ associated to $P_c$.
\end{definition}

Note that this definition does not guarantee the existence of a (strict) high-voltage solution for a given $P_c$.
For such operating points the following proposition applies.
\begin{corollary}\label{corollary:high-voltage solution point dissipation-minimizing}
If the operating point $\tilde V_L$ associated to the power demands $P_c$ is a high-voltage solution, then $\tilde V_L$ is dissipation-minimizing. Moreover, if $\tilde V_L$ is a strict high-voltage solution, then $\tilde V_L$ is the unique dissipation-minimizing operating point.
\end{corollary}
Definitions~\ref{definition:long-term voltage stable} and \ref{definition:long-term voltage semi-stable} describe local properties of an operating point, while Definitions~\ref{definition:dissipation-minimizing operating point} and \ref{definition:high-voltage solution} are global properties concerning all operating points associated to $P_c$.
In Part II of this paper we prove that for each feasible power demand $\tilde P_c$, any operating point associated to $\tilde P_c$ which is either long-term voltage semi-stable, dissipation-minimizing or a high-voltage solution, satisfies all three properties and is unique (\mainresult{desirable operating point}).

\section{Nonnegative feasible power demands}\label{section:nonnegative power demands}
In this section we study the feasibility of nonnegative power demands (\ie, power demands $\tilde P_c$ such that $\tilde P_c\ge \mb 0$).
Recall that in Part~I we consider constant-power loads which could both drain power and inject power.
However, practical applications of DC power grids often deal with constant-power loads that do not inject power into the network, in which case the power demands are nonnegative. %
The goal of this section is to refine the result of Part~I for such power demands.
In particular we show that the necessary and sufficient LMI condition for the feasibility of a vector of power demands $\tilde P_c\in\RR^n$ (Theorem~\ref{theorem:necessary and sufficient condition}) can be refined, leading to a condition which is cheaper to compute. %

This section is structured as follows.
We first identify the operating points corresponding to a nonnegative power demand (Lemma~\ref{lemma:nonnegative power demands}). In addition we present a refinement for the geometric characterization of Theorem~\ref{theorem:convexity of F} (Lemma~\ref{lemma:geometric characterization nonnegative F}), which motivates us to study the boundary of $\mc F$ in more detail.
To study this boundary we deduce an alternative parametrization of $\mc D$ (Theorem~\ref{theorem:alternative parametrization of D}), which is in a sense dual to the parametrization in Theorem~\ref{theorem:parametrization of D}. %
We subsequently give a parametrization of the boundary of $\mc F$ (Theorem~\ref{theorem:parametrization of boundary of F}).
This parametrization gives rise to a parametrization of the boundary of $\mc F$ in the nonnegative orthant (Theorem~\ref{theorem:parametrization nonnegative boundary of F}).
We then reformulate the geometric characterization (Corollary~\ref{corollary:geometric characterization nonnegative F in N}), and refine the necessary and sufficient LMI condition of Theorem~\ref{theorem:necessary and sufficient condition} for nonnegative power demands (Theorem~\ref{theorem:nonnegative necessary and sufficient condition}).

\subsection{Operating points and a geometric characterization for nonnegative power demands}
We are interested in the nonnegative feasible power demands, which are described by the set $\mc F\cap \mc N$, where
\begin{align*}
\mc N := \set{\nu\in\RR^n}{\nu\ge \mb 0}
\end{align*}
denote the nonnegative vectors.
The next lemma characterizes the operating points which correspond to a nonnegative power demand.
\begin{lemma}\label{lemma:nonnegative power demands}
A feasible power demand $\tilde P_c$ is nonnegative (\ie, $\tilde P_c\in\mc F\cap \mc N$) if and only if %
the operating points $\tilde V_L$ associated to $\tilde P_c$ satisfy %
$Y_{LL} \tilde V_L\le Y_{LL}V_L^* = \mc I_L^*$.
\end{lemma}
\begin{proof}
Since operating points are assumed to be positive, we have $\tilde V_L>\mb 0$. Hence
\begin{align*}
\tilde P_c = P_c(\tilde V_L) = [\tilde V_L]Y_{LL}(V_L^* - \tilde V_L) \ge \mb 0
\end{align*}
if and only if $Y_{LL}(V_L^* - \tilde V_L) = \mc I_L^* - Y_{LL}\tilde V_L \ge \mb 0$, where we used \eqref{eqn:open-circuit voltages}.
\end{proof}

Figure~\ref{figure:two nodes voltage domain} illustrates the location of these operating points in the voltage domain.
\begin{figure}%
  \centering
  \includegraphics[width=1\linewidth]{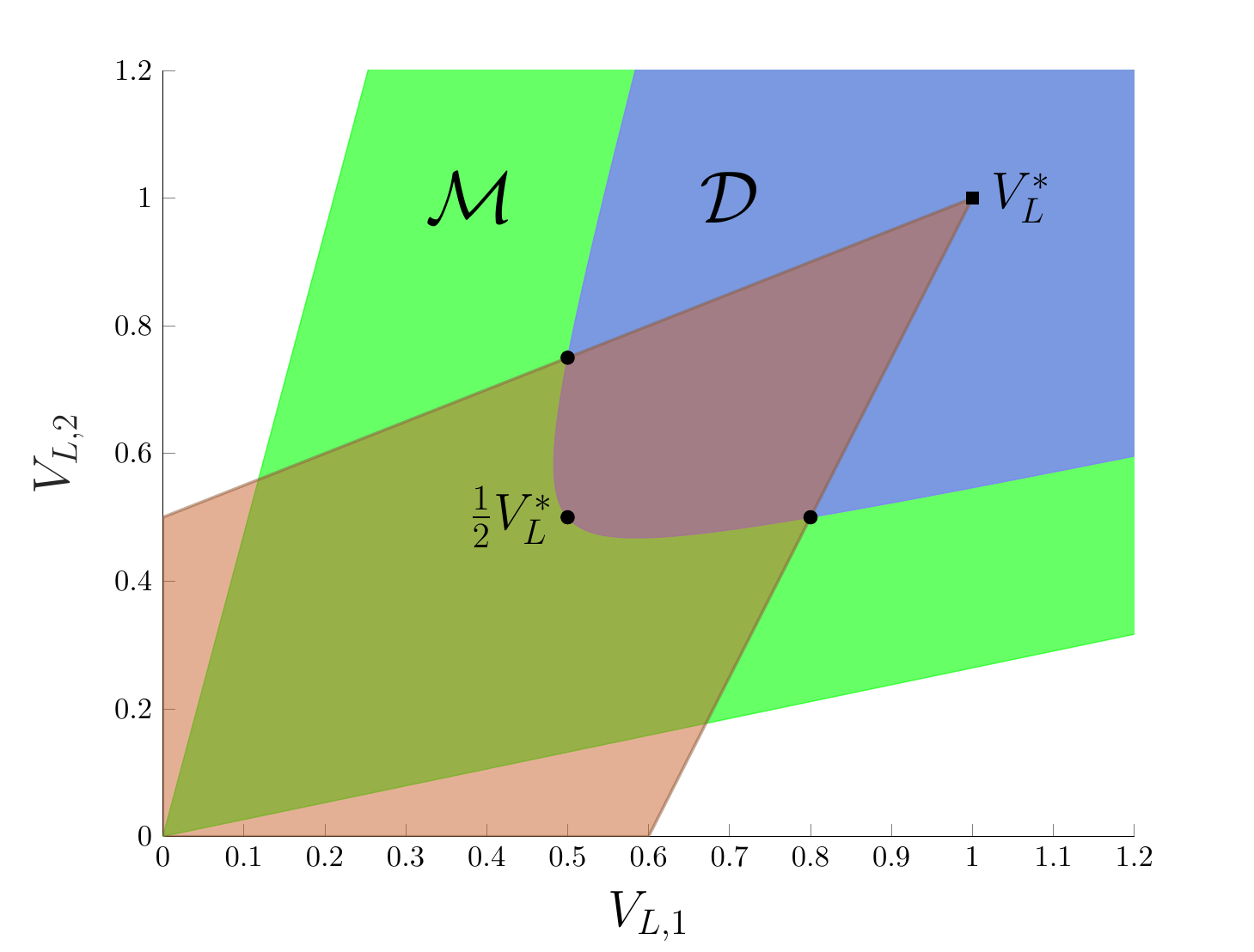}
  \caption{\label{figure:two nodes voltage domain}
A plot of the voltage domain for a power grid with two load nodes. 
The blue area corresponds to the set $\mc D$ of long-term voltage stable operating points.
The brown area indicates the operating points corresponding to a nonnegative power demand.
The green area corresponds to the vectors in $\mc M$, which contains the set $\cl{\mc D}$.
The black operating points correspond to the black power demands in Figure~\ref{figure:polyhedral sufficient condition 2}.}
\end{figure}

Lemma~\ref{lemma:nonnegative power demands} shows that all operating points corresponding to a positive power demand lie in the polyhedral set 
\begin{align}\label{eqn:operating points nonnegative power demand}
\set{\tilde V_L\in\RR^n}{\tilde V_L>\mb 0,~Y_{LL}\tilde V_L\le \mc I_L^*}.
\end{align}
Note that equality holds in $Y_{LL}\tilde V_L\le \mc I_L^*$ if and only if $\tilde V_L=V_L^*$, which corresponds to the power demand $\tilde P_c = \mb 0$.
The next result shows that the vector of open-circuit voltages $V_L^*$ element-wise strictly dominates all operating points corresponding to a nonzero nonnegative power demand. %
\begin{corollary}\label{corollary:domination open-circuit voltages}
Let $\tilde P_c\neq \mb 0$ be a nonnegative feasible power demand, then any operating point $\tilde V_L$ associated to $\tilde P_c$ satisfies $\tilde V_L < V_L^*$. Hence, \eqref{eqn:operating points nonnegative power demand} is bounded.
\end{corollary}
\begin{proof}
The matrix $Y_{LL}$ is an irreducible M-matrix, and hence its inverse is positive by \cite[Thm. 5.12]{fiedler1986special}. 
By Lemma~\ref{lemma:nonnegative power demands} we have $Y_{LL}(V_L^* - \tilde V_L) \ge \mb 0$. 
Since $P_c(V_L^*)=\mb 0$ and $\tilde P_c \neq \mb 0$ it follows that $\tilde V_L \neq V_L^*$ and therefore 
$Y_{LL}(V_L^* - \tilde V_L) \gneqq \mb 0$.
Multiplying this inequality by the positive matrix ${Y_{LL}}\inv$ implies that $V_L^* - \tilde V_L > \mb 0$.
\end{proof}

Using Theorem~\ref{theorem:convexity of F} we present a geometric characterization of $\mc F\cap \mc N$.
\begin{lemma}\label{lemma:geometric characterization nonnegative F}
The set $\mc F \cap \mc N$ is closed, convex, bounded, and is the intersection over all $\lambda\in\Lambda_1$ of the half-spaces $H_\lambda$ for which $P(\varphi(\lambda))$ is nonnegative, \ie,
\begin{align*}
\mc F\cap\mc N = \mc N \cap \bigcap_{\lambda\in\Lambda_1:\;P_c(\varphi(\lambda))\ge \mb 0} H_\lambda.
\end{align*}
\end{lemma}
\begin{proof}
The set $\mc F\cap \mc N$ is the intersection of closed convex sets, and is therefore closed and convex.
It follows from \eqref{eqn:Pmax half-space} that $\mc F\cap \mc N\subset H_{\mb 1} \cap \mc N$. The set $\mc F\cap \mc N$ is bounded since $H_{\mb 1} \cap \mc N$ is %
bounded.
It follows from Theorem~\ref{theorem:convexity of F} that
\begin{align*}
\mc F\cap\mc N = \mc N \cap \bigcap_{\lambda\in\Lambda_1} H_\lambda.
\end{align*}
Since $\mc F\cap \mc N$ is closed and convex, it coincides with the intersection of its supporting half-spaces (see Section~\PIref{subsection:convex hull of F}). Theorem~\PIref{theorem:supporting half-spaces of F} identifies all supporting half-spaces of $\mc F$, and in particular shows that $P_c(\varphi(\lambda))$ is the unique point of support associated to the half-space $H_\lambda$. By definition, $H_\lambda$ is also a supporting half-space for $\mc F\cap\mc N$ if and only if $P_c(\varphi(\lambda))\in \mc F\cap \mc N$, which is equivalent to $P_c(\varphi(\lambda))\ge \mb 0$.
\end{proof}

The power demands $P_c(\varphi(\lambda))$ for $\lambda\in\Lambda_1$ describe the boundary of $\mc F$ (see Theorem~\ref{theorem:one-to-one boundary} and Theorem~\ref{theorem:parametrization of D}).
Lemma~\ref{lemma:geometric characterization nonnegative F} characterizes all nonnegative feasible power demands in terms of the boundary in the nonnegative orthant (\ie, $\partial F\cap \mc N$). 
In its current form, this requires the identification of all $\lambda$ such that $\lambda\in\Lambda$ and $P_c(\varphi(\lambda))\ge \mb 0$, which is a nontrivial computational problem.
In the remainder of this section we deduce an alternative parametrization of the boundary of $\mc F$ in the nonnegative orthant. This parametrization leads to a more constructive description of all such $\lambda$.

\subsection{An alternative parametrization of $\mc D$}\label{subsection:alternative parametrization}
In order to parametrize the boundary of $\mc F$ in the nonnegative orthant, we study the set $\mc D$ of long-term voltage stable operating points in more detail.
In Part~I of this paper we have parametrized the set $\mc D$ and its boundary by means of the set $\Lambda_1$ (see Theorem~\ref{theorem:parametrization of D}). In the following we present an alternative parametrization of $\mc D$, which is dual to the parametrization in Theorem~\ref{theorem:parametrization of D}, in the sense that we parametrize $\mc D$ by the (right) Perron vector of $-\pdd {P_c}{V_L}(\tilde V_L)$ instead of its transpose $-\pdd {P_c}{V_L}(\tilde V_L)\T$.

We introduce the following definitions. For a vector $\mu\in\RR^n$ we introduce the notation
\begin{align}\label{eqn:definition of g}
g(\mu) := [\mu]Y_{LL} + [Y_{LL} \mu].
\end{align}
Note that $g(\mu)$ is linear in $\mu$, and that for any vector $v$ we have $g(\mu)v = g(v)\mu$.
By using \eqref{eqn:jacobian of P_c} and \eqref{eqn:open-circuit voltages} we observe that
\begin{align}
\pdd {P_c}{V_L}(\tilde V_L) &= [Y_{LL} V_L^*] - [\tilde V_L]Y_{LL} - [Y_{LL} \tilde V_L] \nolabel\\
&= [\mc I_L^*] - g(\tilde V_L).\label{eqn:rewrite of jacobian of P_c}
\end{align}
Analogous to $\Lambda$ we define the set
\begin{align*}
\mc M &:= \set{\mu}{g(\mu)\text{ is a nonsingular M-matrix}}.%
\end{align*}
Appendix~\ref{appendix:properties of M} lists several properties of the set $\mc M$. 
In \nobreak{particular}, Lemma~\ref{lemma:positivity of M} shows that $\mc M$ is an open cone which lies in the positive orthant, and that $\mc M$ is simply connected.

Recall that Z-matrices, M-matrices and irreducible matrices were defined in Definitions~\ref{definition:Z-matrix}-\ref{definition:irreducible matrix}.
Appendix~\PIref{subsection:matrix theory} lists multiple properties of such matrices.
Recall the following proposition from Part I of this paper.

\begin{proposition}[Proposition~3.1 of Part I]\label{proposition:Z-matrix perron result}
Let $A$ be an irreducible Z-matrix. There is a unique eigenvalue $r$ of $A$ with smallest (\ie, ``most negative'') real part. The eigenvalue $r$, known as the \emph{Perron root}, is real and simple. A corresponding eigenvector $v$, known as a \emph{Perron vector}, is unique up to scaling, and can be chosen such that $v>\mb 0$.
\end{proposition}

The next lemma relates the Perron root and Perron vector of the Jacobian of $P_c$ to the set $\mc M$.
\begin{lemma}\label{lemma:jacobian of P_c M-matrix}
Let $r\in\RR$ and $\mu\in\RR^n$ such that $r\ge \mb 0$ and $\mu>\mb 0$. %
The Jacobian $-\pdd {P_c}{V_L}(\tilde V_L)$ is an irreducible M-matrix with Perron root $r$ and Perron vector $\mu$ %
if and only if $g(\mu)$ is an M-matrix (\ie, $\mu\in\mc M$) and $\tilde V_L$ satisfies
\begin{align}\label{eqn:jacobian of P_c M-matrix}
\tilde V_L = g(\mu)\inv [\mu](\mc I_L^* + r \mb 1).
\end{align}
\end{lemma}
\begin{proof}
($\Rightarrow$): 
The matrix $Y_{LL}$ is an irreducible Z-matrix and $\mu>\mb 0$, and so $g(\mu)$ is an irreducible Z-matrix by Propositions~\ref{P1-proposition:diagonal properties} and \ref{P1-proposition:sum of irreducible Z-matrices} of Part I. We let $s$ and $v>\mb 0$ denote respectively the Perron root and Perron vector of $g(\mu)$. The matrix $-\pdd {P_c}{V_L}(\tilde V_L)$ is an M-matrix, and therefore a Z-matrix. Lemma~\PIref{lemma:jacobian Z-matrix} states that $-\pdd {P_c}{V_L}(\tilde V_L)$ is a Z-matrix if and only if $V_L^*>\mb 0$, and so $\tilde V_L> \mb 0$.
Using the fact that $(r,\mu)$ is an eigenpair to $-\pdd {P_c}{V_L} (\tilde V_L)$ and substituting \eqref{eqn:rewrite of jacobian of P_c}, we observe that 
\begin{align*}
r\mu = - \pdd {P_c}{V_L} (\tilde V_L) \mu = g(\tilde V_L)\mu - [\mc I_L^*]\mu = g(\mu)\tilde V_L - [\mu]\mc I_L^*.
\end{align*}
By rearranging terms it follows that 
\begin{align}\label{eqn:jacobian of P_c M-matrix:1}
[\mu]\mc I_L^* + r\mu = g(\mu)\tilde V_L.
\end{align}
Multiplying \eqref{eqn:jacobian of P_c M-matrix:1} by $v\T$ results in 
\begin{align}\label{eqn:jacobian of P_c M-matrix:2}
v\T ([\mu]\mc I_L^* + r\mu) = v\T g(\mu)\tilde V_L = s v\T \tilde V_L.
\end{align}
Note that $\tilde V_L>\mb 0$, $v>\mb 0$, $\mu>\mb 0$, $r\ge 0$ and $\mc I_L^* \gneqq \mb 0$. It follows that the left hand side of \eqref{eqn:jacobian of P_c M-matrix:2} is positive. Since $v\T \tilde V_L$ is also positive, we deduce that the Perron root $s$ is positive. This means that $g(\mu)$ is a nonsingular M-matrix (\ie, $\mu\in\mc M$), and that \eqref{eqn:jacobian of P_c M-matrix} follows from \eqref{eqn:jacobian of P_c M-matrix:1}.

($\Leftarrow$): If $\mu\in\mc M$, then $\mu>\mb 0$ by Lemma~\ref{lemma:positivity of M}. The rest of the proof follows by reversing the steps of the ``$\Rightarrow$''-part.
\end{proof}

Note that \eqref{eqn:jacobian of P_c M-matrix} is invariant under scaling of $\mu$, and since $\mc M$ is a cone we may normalize $\mu$. 
For this purpose we define 
\begin{align*}
\mc M_1 &:= \mc M \cap \set{\mu}{\|\mu\|_1 = 1} = \mc M \cap \set{\mu}{\mb 1\T \mu = 1}.
\end{align*}

Lemma~\ref{lemma:jacobian of P_c M-matrix} and Lemma~\ref{lemma:D m-matrix} give rise to an alternative parametrization of $\mc D$.
\begin{theorem}[\mainresult{alternative parametrization of D}]\label{theorem:alternative parametrization of D}
The set $\mc D$ of long-term voltage stable operating points, its closure $\cl{\mc D}$, and its boundary $\partial\mc D$ are parametrized by
\begin{align*}
\mc D &= \set{g(\mu)\inv [\mu](\mc I_L^* + r \mb 1) }{\mu\in\mc M_1, r>0};\\
\cl{\mc D} &= \set{g(\mu)\inv [\mu](\mc I_L^* + r \mb 1) }{\mu\in\mc M_1, r\ge 0};\\
\partial\mc D &= \set{g(\mu)\inv [\mu]\mc I_L^* }{\mu\in\mc M_1}.
\end{align*}
Furthermore, the map
\begin{align*}
(\mu,r)\mapsto g(\mu)\inv [\mu](\mc I_L^* + r \mb 1)
\end{align*}
from $\mc M_1 \times \RR_{\ge 0}$ to $\cl{\mc D}$ is a bicontinuous map.
\end{theorem}

The proof of Theorem~\ref{theorem:alternative parametrization of D} is analogous to the proof of Theorem~\PIref{theorem:parametrization of D}, and is therefore omitted.

To simplify notation, we define for $\mu\in\mc M$ the map
\begin{align}\label{eqn:definition of psi}
\psi(\mu) := g(\mu)\inv [\mu] \mc I_L^*.
\end{align}
Note that Theorem~\ref{theorem:alternative parametrization of D} implies that $\psi(\mc M_1) = \partial \mc D$, which is a parametrization of the boundary of $\mc D$.

Figure~\ref{figure:two nodes voltage domain} illustrates that $\cl{\mc D}$ is in fact a subset of $\mc M_1$, which is shown is Lemma~\ref{lemma:D subset of M}.

Theorem~\ref{theorem:parametrization of D} and Theorem~\ref{theorem:alternative parametrization of D} present two different parametrizations of $\partial \mc D$. The next lemma relates these two parametrizations, and will be instrumental for identifying which $\lambda\in\Lambda$ satisfy $P_c(\varphi(\lambda))\ge \mb 0$ in Lemma~\ref{lemma:geometric characterization nonnegative F}.
\begin{lemma}\label{lemma:duality boundary D}
Let $\tilde V_L\in\partial \mc D$, then there exist 
\begin{enumerate}
\item a unique vector $\lambda\in\Lambda_1$ such that $\tilde V_L = \varphi(\lambda)$;
\item a unique vector $\mu\in\mc M_1$ such that $\tilde V_L = \psi(\mu)$;
\item a positive scalar $c$ such that
\begin{align}\label{eqn:duality boundary D}
[\lambda] \tilde V_L = c \mu.
\end{align}
\end{enumerate}
Consequently, $\mu$ may be expressed in terms of $\lambda$, and vise versa, by
\begin{align}
\mu &= (\lambda\T \varphi(\lambda))\inv [\lambda]\varphi(\lambda) \in \mc M_1\label{eqn:duality boundary D:mu};\\
\lambda &= (\mb 1\T [\psi(\mu)]\inv \mu)\inv [\psi(\mu)]\inv \mu \in \Lambda_1\label{eqn:duality boundary D:lambda}.
\end{align}
\end{lemma}
\begin{proof}
The existence and uniqueness of $\lambda$ and $\mu$ follows respectively from Theorem~\ref{theorem:parametrization of D} and Theorem~\ref{theorem:alternative parametrization of D}.
Since $Y_{LL}$ is symmetric we have
\begin{align}
-\pdd {P_c}{V_L}(\tilde V_L) [\tilde V_L] &= [\tilde V_L] Y_{LL} [\tilde V_L] + [\tilde V_L][Y_{LL}(\tilde V_L - V_L^*)] \nonumber\\
&= - [\tilde V_L] \pdd {P_c}{V_L}(\tilde V_L)\T.\label{eqn:duality boundary D:2}
\end{align}
Note that $-\pdd {P_c}{V_L}$ and its transpose are singular M-matrices by Lemma~\ref{lemma:D m-matrix}, and are irreducible by Lemma~\PIref{lemma:jacobian Z-matrix} since $\tilde V_L>\mb 0$.
Proposition~\PIref{proposition:perron root M-matrix} states that the kernels of $-\pdd {P_c}{V_L}$ and its transpose are spanned by any of their respective Perron vectors.
Hence, if $\lambda\in\Lambda_1$ is such that $\tilde V_L = \varphi(\lambda)$, then $\lambda$ in a Perron vector of $-\pdd {P_c}{V_L}(\tilde V_L)\T$ by Lemma~\PIref{lemma:characterization pddf M-matrix}. We deduce from \eqref{eqn:duality boundary D:2} that
\begin{align*}
\mb 0 = -[\tilde V_L]\pdd {P_c}{V_L}(\tilde V_L)\T\lambda = -\pdd {P_c}{V_L}(\tilde V_L) [\tilde V_L]\lambda.
\end{align*}
It follows that $[\tilde V_L]\lambda$ spans in the kernel of $-\pdd {P_c}{V_L}(\tilde V_L)$. Lemma~\ref{lemma:jacobian of P_c M-matrix} implies that \eqref{eqn:duality boundary D} holds for some scalar $c$. Since $\tilde V_L>\mb 0$, $\lambda>\mb 0$ and $\mu>\mb 0$ we have $c>\mb 0$. Moreover, since $\mu\T \mb 1 = 1$, multiplying \eqref{eqn:duality boundary D} by $\mb 1\T$ yields $c = \lambda\T \tilde V_L = \lambda\T \varphi(\Lambda)$. By taking $c$ to the other side of \eqref{eqn:duality boundary D} we obtain \eqref{eqn:duality boundary D:mu}.
Similarly, since $\lambda\T \mb 1 = 1$, multiplying \eqref{eqn:duality boundary D} by $\mb 1\T [\tilde V_L]\inv$ yields $1 = \mb 1\T [\tilde V_L]\inv \mu c = \mb 1\T [\psi(\mu)]\inv \mu c$. By multiplying \eqref{eqn:duality boundary D} by $[\psi(\mu)]\inv$ we obtain \eqref{eqn:duality boundary D:lambda}.
\end{proof}

Lemma~\ref{lemma:duality boundary D}, and in particular \eqref{eqn:duality boundary D}, establishes a duality between the two parametrizations of $\partial \mc D$. Note that \eqref{eqn:duality boundary D:mu} and \eqref{eqn:duality boundary D:lambda} describe their correspondence.

\subsection{Two parametrizations of the boundary of $\mc F$}\label{subsection:boundary of F}
We continue by studying parametrizations of the boundary of $\mc F$.
Theorem~\ref{theorem:one-to-one boundary} states that $\partial \mc D$ is in one-to-one correspondence with $\partial \mc F$. Since $\partial \mc D$ is parametrized both by $\varphi(\lambda)$ for $\lambda\in\Lambda_1$ (Theorem~\ref{theorem:parametrization of D}) and by $\psi(\mu)$ for $\mu\in\mc M_1$ (Theorem~\ref{theorem:alternative parametrization of D}), it follows that $\partial \mc F$ can be parametrized as
\begin{align*}
\partial \mc F = \set{P_c(\varphi(\lambda))}{\lambda\in\Lambda_1}= \set{P_c(\psi(\mu))}{\mu\in\mc M_1}.
\end{align*}
The next theorem gives an alternative formulation for both of these parametrizations.

\begin{theorem}[\mainresult{boundary of F}]\label{theorem:parametrization of boundary of F}
Let $\tilde P_c\in\partial F$, then there exist unique vectors $\tilde V_L\in \partial D$ and $\mu\in\mc M_1$ such that $\tilde P_c = P_c(\tilde V_L)$ and $\tilde V_L = \psi(\mu)$.
These vectors satisfy
\begin{align}
\tilde P_c &= [\tilde V_L]^2 [\mu]\inv Y_{LL} \mu.\label{eqn:parametrization of boundary of F:mu}
\end{align}
This implies that the boundary of $\mc F$ is parametrized by
\begin{align}
\partial \mc F =\set{[\psi(\mu)]^2 [\mu]\inv Y_{LL} \mu}{\mu\in\mc M_1}.\label{eqn:parametrization of boundary of F:mu 2}
\end{align}
\end{theorem}
\begin{proof}
The existence and uniqueness of $\tilde V_L$ and $\mu$ follows respectively from Theorem~\ref{theorem:one-to-one boundary} and Theorem~\ref{theorem:alternative parametrization of D}.
By \eqref{eqn:definition of psi}, \eqref{eqn:definition of g} and \eqref{eqn:open-circuit voltages} we have 
\begin{align}
\psi(\mu) &= g(\mu)\inv[\mu]\mc I_L^*\nolabel\\
&= ([\mu]Y_{LL} + [Y_{LL}\mu])\inv [\mu] Y_{LL} V_L^* \label{theorem:parametrization of boundary of F:1}\\
&= V_L^* - ([\mu]Y_{LL} + [Y_{LL}\mu])\inv [Y_{LL}\mu] V_L^*.\nolabel
\end{align}
We deduce that
\begin{multline}\label{eqn:parametrization of boundary of F:2}
[\mu]Y_{LL}(V_L^* - \psi(\mu)) \\= [\mu]Y_{LL} ([\mu]Y_{LL} + [Y_{LL}\mu])\inv [Y_{LL}\mu] V_L^*.
\end{multline}
Observe that for any two square matrices $A,B$ such that $A+B$ is nonsingular we have the identity\footnote{This identity may be verified by adding $A(A+B)\inv A$ to both sides of the equation and simplifying.}
\begin{align}\label{eqn:parametrization of boundary of F:4}
A(A+B)\inv B = B(A+B)\inv A.
\end{align}
Using \eqref{eqn:parametrization of boundary of F:4} with $A=[\mu]Y_{LL}$ and $B = [Y_{LL}\mu]$ in \eqref{eqn:parametrization of boundary of F:2} %
yields
\begin{multline}\label{eqn:parametrization of boundary of F:3}
[\mu]Y_{LL}(V_L^* - \psi(\mu)) \\= [Y_{LL}\mu]([\mu]Y_{LL} + [Y_{LL}\mu])\inv  [\mu] Y_{LL} V_L^*\\=[Y_{LL}\mu]\psi(\mu) = [\psi(\mu)]Y_{LL}\mu,
\end{multline}
where we substituted \eqref{theorem:parametrization of boundary of F:1}.
By \eqref{eqn:parametrization of boundary of F:3} it follows that
\begin{multline*}
P_c (\psi(\mu)) = [\psi(\mu)] Y_{LL}(V_L^* - \psi(\mu)) \\= [\psi(\mu)][\mu]\inv [\psi(\mu)]Y_{LL}\mu = [\psi(\mu)]^2[\mu]\inv Y_{LL}\mu,
\end{multline*}
which proves \eqref{eqn:parametrization of boundary of F:mu 2}. Since $\tilde P_c = P_c(\psi(\mu))$ and $\tilde V_L = \psi(\mu)$ we have
$\tilde P_c =  [\tilde V_L]^2[\mu]\inv Y_{LL}\tilde V_L$,
which proves \eqref{eqn:parametrization of boundary of F:mu}.
\end{proof}

The duality of Lemma~\ref{lemma:duality boundary D} implies the following corollary.
\begin{corollary}[\mainresult{boundary of F}]\label{corollary:alternative boundary of F}
Let $\tilde P_c\in\partial F$, then there exists unique vectors $\tilde V_L\in \partial D$ and $\lambda\in\Lambda_1$ such that $\tilde P_c = P_c(\tilde V_L)$ and $\tilde V_L = \varphi(\lambda)$.
These vectors satisfy
\begin{align}
\tilde P_c &= [\varphi(\lambda)] [\lambda]\inv Y_{LL} [\lambda]\varphi(\lambda).\label{eqn:parametrization of boundary of F:lambda}
\end{align}
This implies that the boundary of $\mc F$ is parametrized by
\begin{align}
\partial \mc F =\set{[\varphi(\lambda)] [\lambda]\inv Y_{LL} [\lambda]\varphi(\lambda)}{\lambda\in\Lambda_1}.\label{eqn:parametrization of boundary of F:lambda 2}
\end{align}
\end{corollary}

\subsection{The boundary of $\mc F$ in the nonnegative orthant}

Theorem~\ref{theorem:parametrization of boundary of F} gives an explicit relation between the boundary of $\mc F$ and the vectors $\mu\in\mc M_1$.
The next lemma characterizes all $\mu\in\mc M_1$ for which the corresponding power demand in $\partial \mc F$ lies in the nonnegative orthant.
\begin{lemma}\label{lemma:nonnegative boundary of F}
Given $\tilde P_c\in \partial \mc F$, let $\tilde V_L \in\partial \mc D$ and $\mu\in\mc M_1$ be the unique vectors so that $\tilde P_c = P_c(\tilde V_L)$ and $\tilde V_L = \psi(\mu)$, then $\tilde P_c\in\mc N$ if and only if $Y_{LL}\mu\in\mc N$. Consequently, the boundary of $\mc F$ in the nonnegative orthant is parametrized by
\begin{align*}
\partial \mc F \cap \mc N = \set{P_c(\psi(\mu))}{\mu\in\mc M_1, Y_{LL}\mu\in \mc N}.
\end{align*}
\end{lemma}
\begin{proof}
The existence and uniqueness of $\tilde V_L$ and $\mu$ follow respectively from Theorem~\ref{theorem:one-to-one boundary} and Theorem~\ref{theorem:parametrization of boundary of F}.
Note that $\tilde V_L>\mb 0$, and $\mu >\mb 0$ by Lemma~\ref{lemma:positivity of M}. Hence, it follows from \eqref{eqn:parametrization of boundary of F:mu} that $\tilde P_c\ge \mb 0$ if and only if $Y_{LL}\mu\ge \mb 0$.
The parametrization follows directly from Theorem~\ref{theorem:parametrization of boundary of F}.
\end{proof}

Lemma~\ref{lemma:nonnegative boundary of F} shows that any power demand $\tilde P_c$ in $\partial \mc F \cap \mc N$ is uniquely associated to the vector $Y_{LL}\mu$ in $\mc N$.
Conversely, we now show that any nonzero vector $\nu$ in $\mc N$ is, up to scaling of $\nu$, is uniquely associated to a power demand in $\partial \mc F \cap \mc N$. We require the following lemma.

\begin{lemma}\label{lemma:nonnegative cone in M}
For each nonzero vector $\nu\in \mc N$ we have ${Y_{LL}}\inv\nu\in\mc M$.
\end{lemma}
\begin{proof}
It suffices to show that $g({Y_{LL}}\inv \nu)$ is a nonsingular M-matrix. Note that 
\begin{align*}
g({Y_{LL}}\inv \nu) = [{Y_{LL}}\inv \nu]Y_{LL} + [\nu].
\end{align*}
The matrix $Y_{LL}$ is a nonsingular irreducible M-matrix, and its inverse is a positive matrix by \cite[Thm. 5.12]{fiedler1986special}.
Since ${\nu\gneqq \mb 0}$ it follows that ${Y_{LL}}\inv \nu>\mb 0$. Hence $[{Y_{LL}}\inv \nu]Y_{LL}$ is a nonsingular M-matrix by Proposition~\ref{P1-proposition:diagonal properties}:\PIref{proposition:diagonal properties:M-matrix multiplication}. Since $\nu\gneqq \mb 0$, Proposition~\ref{P1-proposition:diagonal properties}:\PIref{proposition:diagonal properties:irreducible M-matrix addition} implies that $[{Y_{LL}}\inv \nu]Y_{LL} + [\nu]$ is a nonsingular M-matrix.
\end{proof}

We normalize the nonzero vectors in $\mc N$ by
\begin{align}\label{eqn:definition of N_1}
\begin{aligned}
\mc N_1 &:= \mc N \cap \set{\nu}{\|\nu\|_1=1} \\
&\;= \set{\nu\in\RR^n}{\nu\ge\mb 0,~ \mb 1\T \nu=1}.
\end{aligned}
\end{align}
We remark that $\mc N_1$ is known as the standard $n-1$-simplex.

Lemma~\ref{lemma:nonnegative boundary of F} and Lemma~\ref{lemma:nonnegative cone in M} suggest that each $\nu \in \mc N_1$ is uniquely associated to a vector $\mu\in\mc M_1$ for which the associated power demand $P_c(\psi(\mu))$ is nonnegative. 
Since there is a one-to-one correspondence between $\mc M_1$ and $\Lambda_1$ by Lemma~\ref{lemma:duality boundary D}, this would mean that there is a one-to-one correspondence between $\mc N_1$, and the vectors $\lambda\in\Lambda_1$ for which the associated power demand $P_c(\varphi(\lambda))$ is nonnegative.
To this end we define for nonzero $\nu\in\mc N$ the map
\begin{align}\label{eqn:definition of chi}
\chi(\nu) &:= \l[\psi({Y_{LL}}\inv \nu)\r]\inv {Y_{LL}}\inv \nu\\
&\;= \l[[{Y_{LL}}\inv \nu]\inv g\l({Y_{LL}}\inv \nu\r)\inv [{Y_{LL}}\inv \nu]\mc I_L^*\r]\inv \mb 1.\nonumber
\end{align}
Since $Y_{LL}$ is symmetric we have for all $\mu>\mb 0$ that
\begin{align}\label{eqn:property of g}
[\mu]\inv g(\mu)[\mu] = (Y_{LL} + [\mu]\inv [Y_{LL}\mu])[\mu] = g(\mu)\T
\end{align}
by using \eqref{eqn:definition of g}, and hence $\chi(\nu)$ can also be written as
\begin{align}\label{eqn:definition of chi 2}
\chi(\nu) = \l[g\l({Y_{LL}}\inv \nu\r)\Tinv \mc I_L^*\r]\inv \mb 1.
\end{align}

The following theorem establishes a one-to-one correspondece between the set $\mc N_1$ and the sets $\partial \mc F$, $\partial \mc D$, $\mc M_1$ and $\Lambda_1$ for which their associated power demands are nonnegative. In addition, we present a parametrization of the boundary of $\mc F$ restricted to the nonnegative orthant, in terms of $\mc N_1$.

\begin{theorem}[\mainresult{nonnegative power demands}\textbf{a}]\label{theorem:parametrization nonnegative boundary of F}
There is a one-to-one correspondence between the following sets:
\begin{enumerate}[i)]
\item The nonnegative feasible power demands $\tilde P_c$ on the boundary of $\mc F$ (\ie, $\tilde P_c\in \partial \mc F \cap \mc N$);\label{theorem:parametrization nonnegative boundary of F:1}
\item The operating points $\tilde V_L$ on the boundary of $\mc D$ such that $Y_{LL}\tilde V_L \le \mc I_L^*$;\label{theorem:parametrization nonnegative boundary of F:2}
\item The vectors $\mu\in\mc M_1$ such that $Y_{LL}\mu \in\mc N$;\label{theorem:parametrization nonnegative boundary of F:3}
\item The vectors $\lambda \in \Lambda_1$ such that $Y_{LL}[\lambda]\varphi(\lambda)\in\mc N$;\label{theorem:parametrization nonnegative boundary of F:4}
\item The vectors $\nu\in \mc N_1$.\label{theorem:parametrization nonnegative boundary of F:5}
\end{enumerate}
These correspondences satisfy the equations
\begin{align*}
\begin{aligned}
\tilde P_c &= P_c(\tilde V_L);&\\
\tilde V_L &= \psi(\mu) = \psi({Y_{LL}}\inv \nu)\\&= \varphi(\lambda)= \varphi(\chi(\nu));
\end{aligned}\begin{aligned}
\mu &\propto [\lambda]\varphi(\lambda) \propto {Y_{LL}}\inv \nu;\\
\lambda &\propto [\psi(\mu)]\inv \mu \propto \chi(\nu);\\
\nu &\propto Y_{LL}[\lambda]\varphi(\lambda) \propto {Y_{LL}}\mu;
\end{aligned}
\end{align*}
where by $\propto$ we mean that equality holds up to a positive scaling factor. 
In particular, $\chi$ is a one-to-one correspondence between $\mc N_1$ and $\Lambda_1$, up to scaling.
Moreover, the boundary of $\mc F$ in the nonnegative orthant is parametrized by
\begin{align*}
\partial \mc F\cap \mc N = \set{P_c(\psi({Y_{LL}}\inv\nu))}{\nu\in\mc N_1},%
\end{align*}
and the corresponding operating points are parametrized by
\begin{align*}
\set{\tilde V_L\in\partial \mc D}{P_c(\tilde V_L)\ge \mb 0} = \set{\psi({Y_{LL}}\inv\nu))}{\nu\in\mc N_1}.%
\end{align*}
\end{theorem}
\begin{proof}
(\ref{theorem:parametrization nonnegative boundary of F:1} $\leftrightarrow$ \ref{theorem:parametrization nonnegative boundary of F:2}): The map $P_c$ from $\partial \mc D$ to $\partial \mc F$ is a one-to-one by Theorem~\ref{theorem:one-to-one boundary}.
Lemma~\ref{lemma:nonnegative power demands} therefore implies that the map $P_c$ from \ref{theorem:parametrization nonnegative boundary of F:1}) and \ref{theorem:parametrization nonnegative boundary of F:2}) is one-to-one.

(\ref{theorem:parametrization nonnegative boundary of F:1} $\leftrightarrow$ \ref{theorem:parametrization nonnegative boundary of F:3}): The map $\psi$ from $\mc M_1$ to $\partial \mc D$ is one-to-one by Theorem~\ref{theorem:alternative parametrization of D}, and hence $P_c\circ \psi$ from $\mc M_1$ to $\partial \mc F$ is one-to-one.
Lemma~\ref{lemma:nonnegative boundary of F} therefore implies that the map $P_c\circ \psi$ from \ref{theorem:parametrization nonnegative boundary of F:1}) to \ref{theorem:parametrization nonnegative boundary of F:3}) is one-to-one.

(\ref{theorem:parametrization nonnegative boundary of F:3} $\leftrightarrow$ \ref{theorem:parametrization nonnegative boundary of F:4}): 
Lemma~\ref{lemma:duality boundary D} establishes that $\mc M_1$ and $\Lambda_1$ are in one-to-one correspondence, and that $\tilde V_L = \psi(\mu) = \varphi(\lambda)$. Note that \eqref{eqn:duality boundary D:mu} and \eqref{eqn:duality boundary D:lambda} imply that $\mu \propto [\lambda]\varphi(\lambda)$ and $\lambda \propto [\psi(\mu)]\inv \mu$.
Substituting \eqref{eqn:duality boundary D} in \ref{theorem:parametrization nonnegative boundary of F:3}) results in \ref{theorem:parametrization nonnegative boundary of F:4}) and are therefore equivalent.

(\ref{theorem:parametrization nonnegative boundary of F:5} $\leftrightarrow$ \ref{theorem:parametrization nonnegative boundary of F:3}):
Lemma~\ref{lemma:nonnegative cone in M} shows that the map $v\mapsto(\mb 1\T {Y_{LL}}\inv \nu)\inv {Y_{LL}}\inv \nu$ is a map $\mc N_1$ to $\mc M_1$. This map is injective since ${Y_{LL}}\inv$ is nonsingular, and is therefore one-to-one on its image, which is exactly the set \ref{theorem:parametrization nonnegative boundary of F:3}). This shows that $\mu\propto {Y_{LL}}\inv\nu$ and $\nu\propto Y_{LL}\mu$.

Since $\mu \propto [\lambda]\varphi(\lambda)$ and $\nu\propto Y_{LL}\mu$, it follows that $\nu\propto Y_{LL}[\lambda]\varphi(\lambda)$. 
Due to \eqref{eqn:definition of psi} and \eqref{eqn:property of g} we have
\begin{align*}
[\mu]\inv \psi(\mu) = [\mu]\inv g(\mu)\inv [\mu]\mc I_L^* = g(\mu)\invT \mc I_L^*.
\end{align*}
Since $\lambda \propto [\psi(\mu)]\inv \mu$, it follows that $\lambda\propto [g(\mu)\invT \mc I_L^*]\inv \mb 1$.
Because $\mu\propto {Y_{LL}}\inv \nu$, we deduce that $\lambda \propto \chi(\nu)$ by \eqref{eqn:definition of chi 2}.
Thus, the map $\chi$ from $\mc N_1$ to $\Lambda$ is one-to-one, up to scaling.

Finally, the parametrizations follow directly from (\ref{theorem:parametrization nonnegative boundary of F:1} $\leftrightarrow$ \ref{theorem:parametrization nonnegative boundary of F:3}) and (\ref{theorem:parametrization nonnegative boundary of F:5} $\leftrightarrow$ \ref{theorem:parametrization nonnegative boundary of F:3}).
\end{proof}
\begin{remark}\label{remark:computation of boundary}
From a computation standpoint, the parametrization of $\partial\mc F\cap \mc N$ in Theorem~\ref{theorem:parametrization nonnegative boundary of F} is cheaper to compute than the parametrizations of $\partial\mc F$ in Theorem~\ref{theorem:parametrization of boundary of F} or Corollary~\ref{corollary:alternative boundary of F}.
Indeed, to compute the set $\partial \mc F$ we require to identify either $\mc M_1$ or $\Lambda_1$ by Theorem~\ref{theorem:parametrization of boundary of F} or Corollary~\ref{corollary:alternative boundary of F}, respectively, which both are sets that are (in essence) described by the eigenvalues of $n\times n$ matrices.
In contrast, the parametrization of $\partial\mc F\cap \mc N$ in Theorem~\ref{theorem:parametrization nonnegative boundary of F} is in terms of the set $\mc N_1$, which is merely %
an $n-1$-simplex and requires no additional computation.
\end{remark}

\subsection{Refined results for nonnegative power demands}
We conclude this section by presenting a refinement of Theorem~\ref{theorem:convexity of F} and Theorem~\ref{theorem:necessary and sufficient condition} for nonnegative power demands.
This is obtained by applying Theorem~\ref{theorem:parametrization nonnegative boundary of F} to Lemma~\ref{lemma:geometric characterization nonnegative F}.

Theorem~\ref{theorem:parametrization nonnegative boundary of F} states that the map $\chi$ is a one-to-one correspondence between the set $\nu\in\mc N_1$ and vectors $\lambda\in\Lambda_1$ for which the associated power demand $P_c(\varphi(\lambda))$ is nonnegative. More specifically, we have
\begin{multline}\label{eqn:parametrization of Lambda nonnegative powers}
\set{\lambda\in\Lambda_1}{P_c(\varphi(\lambda)) \ge \mb 0} \\= \set{(\mb 1\T \chi(\nu))\inv \chi(\nu)}{\nu\in\mc N_1}\subset \Lambda_1. 
\end{multline}
By substituting this result in Lemma~\ref{lemma:geometric characterization nonnegative F} we obtain a geometric characterization of $\mc F$ in terms of $\mc N_1$.
\begin{corollary}\label{corollary:geometric characterization nonnegative F in N}
The set $\mc F \cap \mc N$ is the intersection over all $\nu\in\mc N_1$ of the half-spaces $H_{\chi(\nu)}$, \ie,
\begin{align*}
\mc F\cap\mc N = \mc N \cap \bigcap_{\nu\in\mc N_1} H_{\chi(\nu)}.
\end{align*}
\end{corollary}
\begin{proof}
The statement follows from substituting \eqref{eqn:parametrization of Lambda nonnegative powers} in Lemma~\ref{lemma:geometric characterization nonnegative F}, and by noting the half-spaces $H_\lambda$ are invariant under scaling of $\lambda$.
\end{proof}

We may now present a necessary and sufficient condition for a vector of nonnegative power demands to be feasible. This condition can be regarded as a refinement of Theorem~\ref{theorem:necessary and sufficient condition} for nonnegative power demands, and is obtained from Corollary~\ref{corollary:geometric characterization nonnegative F in N} by rewriting the half-spaces $H_{\chi(\nu)}$.
\begin{theorem}[\mainresult{nonnegative power demands}\textbf{b}]\label{theorem:nonnegative necessary and sufficient condition}
Let $\tilde P_c$ be a nonnegative power demand (\ie, $\tilde P_c\in \mc N$), then $\tilde P_c$ is feasible (\ie, $\tilde P_c \in \mc F\cap \mc N$) if and only if
\begin{align}
\chi(\nu)\T \tilde P_c \le %
\tfrac 1 2 \nu\T V_L^*\label{eqn:nonnegative necessary and sufficient condition}
\end{align}
for all $\nu\in\mc N_1$, where $\chi(\nu)$ was defined in \eqref{eqn:definition of chi}, where $V_L^*$ are the open-circuit voltages \eqref{eqn:open-circuit voltages}, and where $\mc N_1$ is the standard $n-1$-simplex \eqref{eqn:definition of N_1}. More explicitly, \eqref{eqn:nonnegative necessary and sufficient condition} is equivalent to
\begin{multline*}
\mb 1\T \l[\l([{Y_{LL}}\inv \nu]+ {Y_{LL}} \inv[\nu]\r)\inv V_L^*\r]\inv \tilde P_c %
\le %
\tfrac 1 2 \nu\T V_L^*.
\end{multline*}
Similarly, $\tilde P_c$ is feasible under small perturbation (\ie, $\tilde P_c \in \inter{\mc F}\cap \mc N$) if and only if the inequality in \eqref{eqn:nonnegative necessary and sufficient condition} holds strictly for all $\nu\in\mc N_1$.
\end{theorem}
\begin{proof}
Corollary~\ref{corollary:geometric characterization nonnegative F in N} implies that $\tilde P_c\in \mc F \cap \mc N$ if and only if $\tilde P_c\in\mc N$ and $\tilde P_c\in H_{\chi(\nu)}$ for all $\nu\in\mc N_1$.
By definition of $H_\lambda$, the latter is equivalent to
\begin{align}\label{eqn:nonnegative necessary and sufficient condition:1}
\chi(\nu)\T \tilde P_c \le \|\varphi(\chi(\nu))\|_{h(\chi(\nu))}^2
\end{align}
for all $\nu\in\mc N_1$
We continue by rewriting the right-hand side of \eqref{eqn:nonnegative necessary and sufficient condition:1}. Note that
\begin{align}
\|\varphi(\chi(\nu))\|_{h(\chi(\nu))}^2 &=\varphi(\chi(\nu))\T h(\chi(\nu))\varphi(\chi(\nu))\nolabel\\
&=\tfrac 1 2 \varphi(\chi(\nu))\T [\chi(\nu)]\mc I_L^*,\label{eqn:nonnegative necessary and sufficient condition:2}
\end{align}
where we substituted \eqref{eqn:definition of norm} and \eqref{eqn:definition of phi}.
By substituting \eqref{eqn:definition of chi} in \eqref{eqn:nonnegative necessary and sufficient condition:2} it follows that the right-hand side of \eqref{eqn:nonnegative necessary and sufficient condition:1} equals
\begin{align}\label{eqn:nonnegative necessary and sufficient condition:4}
\tfrac 1 2 \varphi(\chi(\nu))\T [\psi({Y_{LL}}\inv \nu)]\inv [{Y_{LL}}\inv\nu]\mc I_L^*.
\end{align}
Theorem~\ref{theorem:parametrization nonnegative boundary of F} states that $\varphi(\chi(\nu)) = \psi({Y_{LL}}\inv \nu)$, and hence from \eqref{eqn:nonnegative necessary and sufficient condition:4} we deduce that the right-hand side of \eqref{eqn:nonnegative necessary and sufficient condition:1} equals
\begin{align*}
\tfrac 1 2 \mb 1\T [{Y_{LL}}\inv\nu]\mc I_L^* = \tfrac 1 2 \nu\T {Y_{LL}}\inv \mc I_L^* = \tfrac 1 2 \nu\T V_L^*.
\end{align*}
where we used \eqref{eqn:open-circuit voltages}.
The left-hand side of \eqref{eqn:nonnegative necessary and sufficient condition} can be rewritten by observing in \eqref{eqn:definition of chi 2} that
\begin{align*}
g\l({Y_{LL}}\inv \nu\r)\Tinv \mc I_L^* &= (Y_{LL}[{Y_{LL}}\inv \nu] + [\nu])\inv \mc I_L^*\\
&= ([{Y_{LL}}\inv \nu] + {Y_{LL}}\inv[\nu])\inv {Y_{LL}}\inv\mc I_L^*\\
&= ([{Y_{LL}}\inv \nu] + {Y_{LL}}\inv[\nu])\inv V_L^*,
\end{align*}
where we used \eqref{eqn:definition of g} and \eqref{eqn:open-circuit voltages}.

Lemma~\PIref{lemma:lambda f} states that we have equality in \eqref{eqn:nonnegative necessary and sufficient condition:1} if and only if $\tilde P_c = P_c(\varphi(\chi(\nu)))$. Theorem~\ref{theorem:parametrization nonnegative boundary of F} implies that $\tilde P_c\in\partial \mc F\cap \mc N$ if and only if there exists $\nu\in\mc N_1$ such that $\tilde P_c = P_c(\varphi(\chi(\nu)))$. Hence, $\tilde P_c\not\in\partial\mc F\cap \mc N$ if and only if equality in \eqref{eqn:nonnegative necessary and sufficient condition} does not hold for all $\nu\in\mc N_1$. Thus, $\tilde P_c\in\inter{\mc F}\cap \mc N$ if and only if the inequality in \eqref{eqn:nonnegative necessary and sufficient condition} holds strictly for all $\nu\in\mc N_1$. 
\end{proof}

Similar to Remark~\ref{remark:computation of boundary}, we note that the necessary and sufficient condition for feasibility of nonnegative power demands in Theorem~\ref{theorem:nonnegative necessary and sufficient condition} is cheaper to compute than the LMI condition in Theorem~\ref{theorem:necessary and sufficient condition}.
Indeed, Theorem~\ref{theorem:necessary and sufficient condition} relies on the computation of the definiteness of an $n+1\times n+1$-matrix. Similar to the proof of Theorem~\ref{theorem:necessary and sufficient condition}, it can be shown  by means of the Haynsworth inertia additivity formula that this computation is equivalent to identifying set $\Lambda_1$ and verifying that the inequality
\begin{align}\label{eqn:haynsworth component}
\lambda\T \tilde P_c \le \tfrac 1 4 (\mc I_L^*)\T[\lambda] h(\lambda)\inv [\lambda]\mc I_L^* = \|\varphi(\lambda)\|_{h(\lambda)}^2
\end{align}
holds for all $\lambda\in\Lambda_1$. In contrast, Theorem~\ref{theorem:nonnegative necessary and sufficient condition} proves that for a nonnegative power demand it is sufficient to consider only vectors $\lambda$ in the subset of $\Lambda_1$ described by \eqref{eqn:parametrization of Lambda nonnegative powers}, which is parametrized by the simplex $\mc N_1$.
Since $\chi(\nu)\in\Lambda_1$ for all $\nu\in\mc N_1$, it is not necessary to compute $\Lambda_1$ to decide the feasibility of nonnegative power demands.
Alternatively we note that, since the set $\Lambda_1$ is convex by Lemma~\PIref{lemma:Lambda convex}, it is sufficient to verify \eqref{eqn:haynsworth component} for all $\lambda$ in the convex hull of the set \eqref{eqn:parametrization of Lambda nonnegative powers} in order to decide if a nonnegative power demand is feasible.

\begin{remark}
Similar results for positive power demand are obtained by taking $\mc N = \set{\nu\in\RR^n}{\nu>\mb 0}$ throughout this section. In particular, analogous to Theorem~\ref{theorem:nonnegative necessary and sufficient condition}, it can be shown that a vector of positive power demands $\tilde P_c>\mb 0$ is feasible if and only if \eqref{eqn:nonnegative necessary and sufficient condition} holds for all $\nu>\mb 0$, and similar for feasibility under small perturbation.
\end{remark}

\section{Sufficient conditions for power flow feasibility}\label{section:sufficient conditions}
In the remainder of this paper we return to the case where power demands are not restricted to the nonnegative orthant.
In this section we prove two novel sufficient conditions for the feasibility of a vector of power demands,
which generalize the sufficient conditions found in \cite{bolognani2015existence} and \cite{simpson2016voltage}. In addition we show how the conditions in \cite{bolognani2015existence} and \cite{simpson2016voltage} are recovered from the conditions proposed in this section.

The benefit of these sufficient conditions for power flow feasibility
over a necessary and sufficient condition such as Theorem~\ref{theorem:necessary and sufficient condition} is that they are cheaper to compute, and may therefore be more applicable in practical applications.
However, since these sufficient condition are not necessary, they %
cannot guarantee that a power demand is not feasible.

This section is structured as follows. First we show that each feasible vector of power demands gives rise to a sufficient condition for power flow feasibility (Lemma~\ref{lemma:power demand domination}), and derive a sufficient condition from Theorem~\ref{theorem:nonnegative necessary and sufficient condition} (Corollary~\ref{corollary:sufficient condition nonnegative}). 
Next, we propose a weaker sufficient condition (Theorem~\ref{theorem:generalization simpson-porco condition}), which generalizes the condition in \cite{simpson2016voltage} and identifies for which vectors the latter condition is tight. 
Finally we show that Theorem~\ref{theorem:generalization simpson-porco condition} generalizes the sufficient condition in \cite{bolognani2015existence}, and argue why the latter condition is not tight in general (Lemma~\ref{lemma:bolognani boundary}).

\subsection{Sufficient conditions by element-wise domination}
\begin{lemma}[\mainresult{power demand domination}]\label{lemma:power demand domination}
Let $\tilde P_c$ be a feasible power demand (\ie, $\tilde P_c \in \mc F$). If a power demand $\hat P_c$ satisfies $\hat P_c\lneqq\tilde P_c$, then $\hat P_c$ is feasible under small perturbation (\ie, $\hat P_c \in\inter{\mc F}$).
\end{lemma}
\begin{proof}
Since $\tilde P_c \in \mc F$ we have by Theorem~\ref{theorem:convexity of F} that %
\begin{align*}
\lambda\T \tilde P_c \le \|\varphi(\lambda)\|_{h(\lambda)}^2
\end{align*}
for all $\lambda\in\Lambda$, where we used \eqref{eqn:definition of H}.
Note that $\lambda>\mb 0$ for $\lambda\in\Lambda$ by Lemma~\PIref{lemma:Lambda positive}.
Since $\hat P_c\lneqq\tilde P_c$ we have 
\begin{align}\label{eqn:power demand domination:1}
\lambda\T \hat P_c < \lambda\T \tilde P_c \le \|\varphi(\lambda)\|_{h(\lambda)}^2,
\end{align}
for all $\lambda\in\Lambda$. Hence, $\hat P_c\in\mc F$ by Theorem~\ref{theorem:convexity of F} and \eqref{eqn:definition of H}. Since the inequality in \eqref{eqn:power demand domination:1} is strict, Lemma~\PIref{eqn:lambda f inequality} implies that $\hat P_c\neq P_c(\varphi(\lambda))$ for all $\lambda\in\Lambda$, and therefore $\hat P_c \not\in\partial\mc F$ by Theorem~\ref{theorem:one-to-one boundary} and Theorem~\ref{theorem:parametrization of D}. 
\end{proof}

Lemma~\ref{lemma:power demand domination} shows that any feasible power demand gives rise to a sufficient condition for power flow feasibility. In particular, note that the power demand $\mb 0 = P_c(V_L^*)$ is feasible under small perturbation. Lemma~\ref{lemma:power demand domination} therefore implies the following corollary.
\begin{corollary}\label{corollary:nonnegative power demands are feasible}
Any nonpositive power demand is feasible under small perturbation.
\end{corollary}

We remark that a vector of nonpositive power demands corresponds to a case in which none of the power loads drain power from the grid and therefore behave as sources. Intuitively it is clear that such a vector of power demands is feasible. Consequently, some of the sources may act as loads and drain the power that is not dissipated in the lines. %

Recall that Theorem~\ref{theorem:nonnegative necessary and sufficient condition} gives a necessary and sufficient condition for the feasibility of a nonnegative power demand.
Lemma~\ref{lemma:power demand domination} allows us to extend Theorem~\ref{theorem:nonnegative necessary and sufficient condition} to a sufficient condition for vectors of power demands which have negative entries.
We define $\mathbf{max}(a,b)\in\RR^n$ as the vector obtained by taking the element-wise maximum of $a,b\in\RR^n$, \ie,
\begin{align*}
\mathbf{max}(a,b)_i := \max(a_i,b_i).
\end{align*}
Note for $\tilde P_c\in \RR^n$ that $\mathbf{max}(\tilde P_c,\mb 0)$ is nonnegative, and that $\tilde P_c\le \mathbf{max}(\tilde P_c,\mb 0)$.
Hence, Theorem~\ref{theorem:nonnegative necessary and sufficient condition} and Lemma~\ref{lemma:power demand domination} directly imply the following sufficient condition for the feasibility of a vector of power demands.
\begin{corollary}[\mainresult{sufficient conditions}]\label{corollary:sufficient condition nonnegative}
A vector of power demands $\tilde P_c\in \RR^n$ is feasible (\ie, $\tilde P_c \in \mc F$) if 
\begin{align*}
\chi(\nu)\T \mathbf{max}(\tilde P_c,\mb 0) \le
\tfrac 1 2 \nu\T V_L^*
\end{align*}
for all $\nu\in\mc N_1$.
\end{corollary}

Note that Corollary~\ref{corollary:sufficient condition nonnegative} is necessary and sufficient for nonnegative power demands by Theorem~\ref{theorem:nonnegative necessary and sufficient condition}.

\begin{figure}[t]
\centering
\hspace*{-22pt}
\includegraphics[scale=.6]{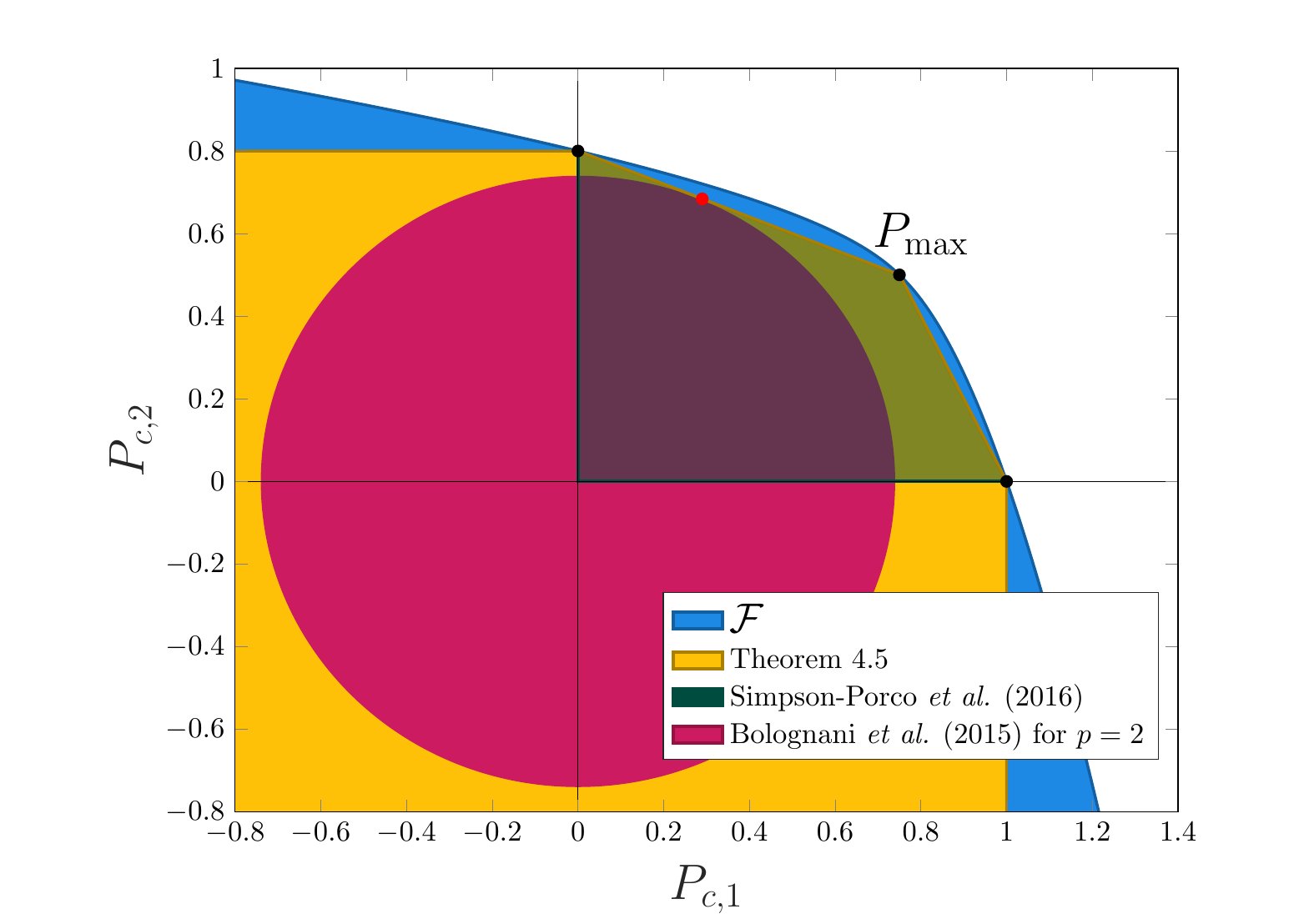}
\caption{\label{figure:polyhedral sufficient condition 2}
A depiction of the set $\mc F$ of feasible power demands for a power grid with two loads. %
The yellow area corresponds to the sufficient condition in Theorem~\ref{theorem:generalization simpson-porco condition}.
The green shaded area corresponds to the %
set described by the sufficient condition in \cite{simpson2016voltage} (see Proposition~\ref{proposition:simpson-porco condition}), and does not include the yellow boundary. The black points are the power demands for which the condition in \cite{simpson2016voltage} is tight, and correspond to the black operating points  in Figure~\ref{figure:two nodes voltage domain}. The red area corresponds to the sufficient condition in \cite{bolognani2015existence} (see Proposition~\ref{proposition:bolognani condition}). The red point indicates a point of intersection of the boundary of the condition in \cite{bolognani2015existence} with either the boundary of the condition in \cite{simpson2016voltage}, or the boundary of the condition in Theorem~\ref{theorem:generalization simpson-porco condition}.}
\end{figure}
\subsection{A generalization of the sufficient condition of Simpson-Porco \textit{et al.} (2016)}%
We proceed by studying known sufficient conditions in the literature and comparing them to Corollary~\ref{corollary:sufficient condition nonnegative}.
The paper \cite{simpson2016voltage} studies the decoupled reactive power flow equations for lossless AC power grids with constant power loads.
The analysis and results in \cite{simpson2016voltage} translate naturally to DC power grids.
In \cite{simpson2016voltage} a sufficient condition for the feasibility of a vector of constant power demands is proposed, which we state in the context of DC power grids.
\begin{proposition}[{\cite[Supplementary Theorem 1]{simpson2016voltage}}]\label{proposition:simpson-porco condition}
Let $\tilde P_c$ be a nonnegative vector of power demands (\ie, $\tilde P_c\in\mc N$), then $\tilde P_c$ is feasible under small perturbation (\ie, $\tilde P_c\in\inter{\mc F}$) if 
\begin{align}
\|(\tfrac 1 4  [V_L^*]Y_{LL}[V_L^*])\inv \tilde P_c \|_\infty < 1.\label{eqn:simpson-porco condition}
\end{align}
This sufficient condition for feasibility is tight since we have 
\begin{align*}
\|(\tfrac 1 4 [V_L^*]Y_{LL}[V_L^*])\inv P_{\mathrm{max}} \|_\infty = 1,
\end{align*}
where $P_{\mathrm{max}}\in\partial\mc F$ is the maximizing power demand defined in Lemma~\ref{lemma:maximizing power demand}, and lies on the boundary of $\mc F$.
\end{proposition}

Proposition~\ref{proposition:simpson-porco condition} applies only to nonnegative power demands and is not necessary in general. It is therefore weaker than Theorem~\ref{theorem:nonnegative necessary and sufficient condition} and Corollary~\ref{corollary:sufficient condition nonnegative}.
The proof of Proposition~\ref{proposition:simpson-porco condition} in \cite{simpson2016voltage} relies on a fixed point argument.
The following result generalizes Proposition~\ref{proposition:simpson-porco condition}, 
and identifies
all power demands for which the condition \eqref{eqn:simpson-porco condition} is tight (\ie, the power demands on the boundary of $\mc F$ so that equality in \eqref{eqn:simpson-porco condition} holds).

\begin{theorem}[\mainresult{sufficient conditions}]\label{theorem:generalization simpson-porco condition}
A vector of power demands $\tilde P_c$ is feasible (\ie, $\tilde P_c\in\mc F$) if 
\begin{align}\label{eqn:generalization simpson-porco condition}
(\tfrac 1 4 [V_L^*]Y_{LL}[V_L^*])\inv \mathbf{max}(\tilde P_c,\mb 0) \le \mb 1,
\end{align}
and feasible under small perturbation (\ie, $\tilde P_c\not\in\inter{\mc F}$) if $\tilde P_c$ is not of the form
\begin{align}\label{eqn:generalization simpson-porco condition powers}
\begin{aligned}
(\tilde P_c)_{[\alpha]} &= \tfrac 1 4 [(V_L^*)_{[\alpha]}]((Y_{LL} / (Y_{LL})_{[\alpha\comp,\alpha\comp]})(V_L^*)_{[\alpha]}; \\
(\tilde P_c)_{[\alpha\comp]} &=\mb 0
\end{aligned}
\end{align}
for all nonempty $\alpha\subset \boldsymbol n$.
\end{theorem}

The proof of Theorem~\ref{theorem:generalization simpson-porco condition} is found in Appendix~\ref{appendix:proof of sufficient condition}.
Note that Theorem~\ref{theorem:generalization simpson-porco condition} is weaker than Corollary~\ref{corollary:sufficient condition nonnegative}, but is cheaper to compute.
Proposition~\ref{proposition:simpson-porco condition} is recovered from Theorem~\ref{theorem:generalization simpson-porco condition} as follows.
\begin{proofof}{Proposition~\ref{proposition:simpson-porco condition}}
Let $\tilde P_c\in \mc N$ which implies that $\tilde P_c = \mathbf{max}(\tilde P_c,\mb 0)$. Let $\tilde P_c$ satisfy \eqref{eqn:simpson-porco condition}, which is therefore equivalent to
\begin{align}\label{eqn:simpson-porco condition:1}
- \mb 1 < (\tfrac 1 4  [V_L^*]Y_{LL}[V_L^*])\inv \mathbf{max}(\tilde P_c,\mb 0) < \mb 1.
\end{align}
It follows from Theorem~\ref{theorem:generalization simpson-porco condition} that $\tilde P_c\in\mc F$.
The latter inequality in \eqref{eqn:simpson-porco condition:1} is strict, and therefore $\tilde P_c$ lies in the interior of the set described by \eqref{eqn:generalization simpson-porco condition}, and hence $\tilde P_c\in\inter{\mc F}$.
\end{proofof}

Theorem~\ref{theorem:generalization simpson-porco condition} states that the power demands described by \eqref{eqn:generalization simpson-porco condition powers} are the only power demands which satisfy \eqref{eqn:generalization simpson-porco condition} and lie on the boundary of $\mc F$. The condition in \eqref{eqn:simpson-porco condition} is therefore tight for such power demands. 
\begin{remark}
Note that if $\alpha=\boldsymbol n$ in \eqref{eqn:generalization simpson-porco condition powers} we obtain the maximizing power, since 
\begin{align*}
\tfrac 1 4[V_L^*]Y_{LL}[V_L^*]\mb 1 = \tfrac 1 4[V_L^*]\mc I_L^* = P_{\text{max}},
\end{align*}
by \eqref{eqn:open-circuit voltages} and \eqref{eqn:max total power demand}.
The proof of Theorem~\ref{theorem:generalization simpson-porco condition} shows that the the power demands described by \eqref{eqn:generalization simpson-porco condition powers} correspond to the maximizing power demands of all power grids obtained by Kron-reduction (see, \eg, \cite{dorfler2012kron}). The power flow of such power grids is equivalent to power flow of the full power grid, with the additional restriction that the currents at the loads indexed by $\alpha$ vanish (\ie, $(\mc I_L)_{[\alpha]}=\mb 0$).
\end{remark}

\subsection{On the sufficient condition of Bolognani \& Zampieri (2015)}%
We conclude this section by showing that Theorem~\ref{theorem:generalization simpson-porco condition} also generalizes the sufficient condition in \cite{bolognani2015existence}.
The paper \cite{bolognani2015existence} studies the power flow equation of an AC power grid with constant power loads and a single source node. The analysis and results in \cite{bolognani2015existence} translate naturally to DC power grids with a single source node. 
The next lemma show that the results in \cite{bolognani2015existence} apply to DC power grids with multiple sources as well, which allows us to compare Theorem~\ref{theorem:generalization simpson-porco condition} and \cite{bolognani2015existence}.
\begin{lemma}\label{lemma:single source}
Let $\mc P$ denote a DC power grid with constant power loads with $n$ loads and $m$ sources as described in Section~\ref{section:summary of part 1}. Let $\hat{\mc P}$ denote the DC power grid with $n$ loads and a single source, of which the Kirchhoff matrix satisfies
\begin{align}\label{eqn:single source}
\hat Y = \begin{pmatrix}
\hat Y_{LL} & \hat Y_{LS} \\ \hat Y_{SL} & \hat Y_{SS}
\end{pmatrix} = \begin{pmatrix}
[V_L^*]Y_{LL}[V_L^*] & - [V_L^*]\mc I_L^* \\ - (\mc I_L^*)\T [V_L^*] & (V_L^*) \T \mc I_L^*
\end{pmatrix}.%
\end{align}
and where the source voltage equals $\hat V_S=1$.
The feasibility of the power flow equations of $\mc P$ and $\hat{\mc P}$ is equivalent.
\end{lemma}
\begin{proof}
We first verify that $\hat Y$ is indeed a Kirchhoff matrix.
Note that $[V_L^*]Y_{LL}[V_L^*]\mb 1 = [V_L^*]\mc I_L^*$ by \eqref{eqn:open-circuit voltages}, and so $\hat Y\mb 1 = \mb 0$. Also, since $Y_{LL}$ is an irreducible Z-matrix, and $V_L^*>\mb 0$ and $\mc I_L^*\gneqq \mb 0$, $\hat Y$ is also an irreducible Z-matrix, and therefore a Kirchhoff matrix.
The powers injected at the loads in power grid $\hat{\mc P}$ satisfy
\begin{align*}
\hat P_L(\hat V_L) &= [\hat V_L](\hat Y_{LL} \hat V_L + \hat Y_{LS} \hat V_S) \\
&= [\hat V_L]([V_L^*]Y_{LL}[V_L^*] \hat V_L - [V_L^*]\mc I_L^*) \\
&= [\hat V_L] [V_L^*] (Y_{LL} [V_L^*] \hat V_L - Y_{LL} V_L^*) 
= P_L([V_L^*]\hat V_L).
\end{align*}
where we used \eqref{eqn:open-circuit voltages} and \eqref{eqn:power at the nodes}. We therefore have $\hat P_L(\hat V_L) = P_L(V_L)$ by taking $V_L = [V_L^*]\hat V_L$. Hence, given $P_c\in\RR^n$, we have that $\hat P_L(\hat V_L) = P_c$ is feasible for some $\hat V_L>\mb 0$ if and only if $P_L(V_L) = P_c$ is feasible for some $V_L>\mb 0$.
\end{proof}

We continue by formulating the sufficient condition in \cite{bolognani2015existence}.
We follow \cite{bolognani2015existence} and define for $p\in[1,\infty]$ the matrix norm
\begin{align}
\|A\|_p^\star: = \max_j\{\|A_{[j, \boldsymbol n]}\|_p\},\label{eqn:bolognani norm}
\end{align}
where $A_{[j, \boldsymbol n]}$ denotes the $j$-th row of $A$.
The sufficient condition for power flow feasibility of \cite{bolognani2015existence} in the context of DC power grids is given as follows.
\begin{proposition}[{\cite[Thm. 1]{bolognani2015existence}}]\label{proposition:bolognani condition}
Let $\tilde P_c\in\RR^n$ be a vector of power demands, then $\tilde P_c$ is feasible under small perturbation (\ie, $\tilde P_c\in\inter{\mc F}$) if for some $p,q\in[1,\infty]$ such that $\tfrac 1 p + \tfrac 1 q = 1$ we have
\begin{align}
\|(\tfrac 1 4 [V_L^*]Y_{LL}[V_L^*])\inv \|_q^\star ~\| \tilde P_c \|_p < 1.\label{eqn:bolognani condition}
\end{align}
\end{proposition}

The proof of Proposition~\ref{proposition:bolognani condition} in \cite{bolognani2015existence} also relies on a fixed point argument.
Proposition~\ref{proposition:bolognani condition} is recovered from Lemma~\ref{lemma:power demand domination} and Proposition~\ref{proposition:simpson-porco condition} as follows.

\begin{proof}
Let $\tilde P_c$ satisfy \eqref{eqn:bolognani condition}. Let $\hat P_c\in\mc N$ be such that $(\hat P_c)_i = |(\tilde P_c)_i|$. It follows that $\|\hat P_c\|_p = \|\tilde P_c\|_p$, and hence $\hat P_c$ satisfies \eqref{eqn:bolognani condition}.
The matrix $Y_{LL}$ is an M-matrix, and hence $[V_L^*]Y_{LL}[V_L^*]$ is a M-matrix by Proposition~\ref{P1-proposition:diagonal properties}.\PIref{proposition:diagonal properties:M-matrix multiplication}.
Its inverse is a positive matrix by \cite[Thm. 5.12]{fiedler1986special}.
Let $v^j\in\RR^n$ be such that $(v^j)\T$ is the $j$-th row of $(\tfrac 1 4 [V_L^*]Y_{LL}[V_L^*])\inv$. We have $v^j>\mb 0$.
By \eqref{eqn:bolognani norm} it follows from \eqref{eqn:bolognani condition} that for all $j$
\begin{align*}
\|v^j\|_q \|\hat P_c\|_p < 1.
\end{align*}
By H\"older's inequality (see, \eg, \cite[pp. 303]{roman2008advanced}) we have
\begin{align*}
\|[v^j]\hat P_c\|_1 \le \|v^j\|_q \|\hat P_c\|_p < 1.
\end{align*}
Since $[v^j]P_c \ge \mb 0$ we know that $\|[v^j]\hat P_c\|_1 = (v^j)\T \hat P_c$. 
This implies that $(v^j)\T \hat P_c< 1$ for all $i$, and hence 
\begin{align*}
(\tfrac 1 4 [V_L^*]Y_{LL}[V_L^*])\inv \hat P_c < \mb 1.
\end{align*}
Hence \eqref{eqn:simpson-porco condition} holds for $\hat P_c$, and $\hat P_c\in\inter{\mc F}$ by Proposition~\ref{proposition:simpson-porco condition}.
Since $\tilde P_c \le \hat P_c$, Lemma~\ref{lemma:power demand domination} implies that $\tilde P_c \in \inter{\mc F}$.
\end{proof}

Our proof of Proposition~\ref{proposition:bolognani condition} shows that the sufficient condition in \cite{bolognani2015existence} for nonnegative power demands is more conservative in comparison to the sufficient condition in \cite{simpson2016voltage}. This also shows that Theorem~\ref{theorem:generalization simpson-porco condition} generalizes both results. 
The next lemma gives a more intuitive interpretation of the condition \eqref{eqn:bolognani condition}, by showing that \eqref{eqn:bolognani condition} describes the largest open $p$-ball such that \eqref{eqn:simpson-porco condition} holds for nonnegative power demands.
\begin{lemma}\label{lemma:bolognani boundary}
Let $p,q\in[1,\infty]$ such that $\tfrac 1 p + \tfrac 1 q = 1$.
The sufficient condition for power flow feability \eqref{eqn:bolognani condition} in Proposition~\ref{proposition:bolognani condition} describes the open ball centered at $\mb 0$
\begin{align*}
\mc B :=\set{y\in \RR^n}{\|y\|_p<r}
\end{align*}
where the radius $r = (\|(\tfrac 1 4 [V_L^*]Y_{LL}[V_L^*])\inv \|_q^\star)\inv>0$ is the largest scalar such that \eqref{eqn:simpson-porco condition} holds for all $\tilde P_c\in\mc B \cap \mc N$.
\end{lemma}
\begin{proof}
We show that there exists a nonnegative vector of power demands on the boundary of $\mc B$ such that such that equality in \eqref{eqn:simpson-porco condition} holds, which therefore defines to the radius $r$.
We continue our proof of Proposition~\ref{proposition:bolognani condition}.
Let $j$ be such that $\|v^j\|_q = \|(\tfrac 1 4 [V_L^*]Y_{LL}[V_L^*])\inv \|_q^\star$. 
If $p\neq 1$, then equality in Hölder's inequality $\|[v^j]\tilde P_c\|_1 \le \|v^j\|_q \|\tilde P_c\|_p$ holds if $(\tilde P_c)_i = c ((v^j)_i)^{q-1}$ for all $i$ and for any $c\in \RR^n$. 
Consider the positive vector of power demands $\hat P_c$ given by $(\hat P_c)_i = c ((v^j)_i)^{q-1}$ and where $c\inv = \|v^j\|_q^q$. For this vector we have $\|[v^j]\hat P_c\|_1 = \|v^j\|_q \|\hat P_c\|_p = 1$.
Hence, by following proof of Proposition~\ref{proposition:bolognani condition}, $\hat P_c$ satisfies equality in both \eqref{eqn:simpson-porco condition} and \eqref{eqn:bolognani condition}. Thus, $\hat P_c\in\partial \mc B\cap \mc N$, and $\|\hat P_c\|_p =  (\|v^j\|_q)\inv = r$.
If $p=1$, the same holds when we take $\hat P_c = e_i \|v^j\|_\infty\inv$, where $i$ is a single index such that $(v^j)_i = \|v^j\|_\infty$, and $(\hat P_c)_i = 0$ otherwise.
\end{proof}

Note that $\hat P_c$ constructed in the proof of Lemma~\ref{lemma:bolognani boundary} is not necessarily of the form \eqref{eqn:generalization simpson-porco condition powers}. Since \eqref{eqn:simpson-porco condition} is tight only for such points, this suggests that the condition \eqref{eqn:bolognani condition} is not tight in general.
This is can be observed for $p=q=2$ in Figure~\ref{figure:polyhedral sufficient condition 2} by the red dot, which %
does not lie on the boundary of $\mc F$.

\section{Desirable operating points}\label{section:high-voltage solution}

We conclude this paper by showing that for each feasible vector of power demands the different definitions of desirable operating points in Part~I (see Definitions~\ref{definition:long-term voltage semi-stable}, \ref{definition:dissipation-minimizing operating point} and \ref{definition:high-voltage solution}) identify the same unique operating point.
It was shown in \cite{matveev2020tool} that for each feasible power demand there exists a unique operating point which is a  high-voltage solution, and that this operating point is ``almost surely'' long-term voltage stable.
In addition, \cite{matveev2020tool} states that this operating point is the unique long-term voltage stable operating point if all power demands have the same sign.
The next theorem sharpens these results by showing that the long-term voltage stable operating point associated to a feasible vector of power demands
is a strict high-voltage solution.
\begin{theorem}[\mainresult{desirable operating point}]\label{theorem:dominating solution}
Let $\tilde P_c$ be a feasible vector of power demands (\ie, $\tilde P_c\in \mc F$). Let $\tilde V_L\in \cl{\mc D}$ be such that $\tilde V_L$ is an operating point associated to $\tilde P_c$ (\ie, $\tilde P_c = P_c(\tilde V_L)$), which exists and is unique by Theorem~\ref{theorem:one-to-one correspondence}.
Suppose there exists a vector $\tilde V_L^\prime\in\RR^n$ such that $\tilde V_L^\prime \neq \tilde V_L$ %
and $\tilde P_c=P_c(\tilde V_L^\prime)$, %
then $\tilde V_L^\prime < \tilde V_L$. Hence, $\tilde V_L$ is a strict high-voltage solution. Moreover, $\tfrac 1 2 (\tilde V_L^\prime + \tilde V_L)$ lies on the boundary of $\mc D$.
\end{theorem}
\begin{proof}
If $\tilde P_c \in \partial \mc F$, then by Theorem~\ref{theorem:one-to-one boundary} the operating point $\tilde V_L\in \partial \mc D$ is the unique operating point associated to $\tilde P_c$. Hence a second operating point $\tilde V_L^\prime$ does not exist. The uniqueness of $\tilde V_L$ implies that $\tilde V_L$ is a high-voltage solution and is dissipation-minimizing.

If $\tilde P_c \in \inter{\mc F}$, then by Corollary~\ref{corollary:interior F one-to-one D} we have $\tilde V_L\in \mc D$. 
We define the vectors $v := \tfrac 1 2 (\tilde V_L + \tilde V_L^\prime)$ and $\mu:=\tfrac 1 2 (\tilde V_L - \tilde V_L^\prime)$, and the line $\gamma(\theta) := v + \theta \mu$.
Note that $\gamma(1) = \tilde V_L$ and $\gamma(-1) = \tilde V_L^\prime$.
Since $\tilde V_L\in \mc D$ and $\tilde V_L^\prime\not\in\cl{D}$ we have $\tilde V_L \neq \tilde V_L^\prime$, and so $\mu \neq \mb 0$.
Lemma~\PIref{lemma:lambda f} implies that
\begin{align}\label{eqn:dominating solution:1}
P_c(\gamma(\theta)) &= P_c(v + \theta \mu)\nolabel \\&= P_c(v) + \theta \pdd {P_c}{V_L} (v)\mu - \theta^2[\mu]Y_{LL}\mu.
\end{align}
Since $\tilde P_c = P_c(\gamma(1)) = P_c(\gamma(-1))$, it follows from \eqref{eqn:dominating solution:1} that 
\begin{align}\label{eqn:dominating solution:3}
\pdd {P_c}{V_L} (v)\mu = \mb 0.
\end{align}
We therefore have 
\begin{align}\label{eqn:dominating solution:2}
P_c(\gamma(\theta)) =  P_c(v) - \theta^2[\mu]Y_{LL}\mu,
\end{align}
which describes a half-line contained in $\mc F$.
Note also that $P_c(\gamma(\theta)) = P_c(\gamma(-\theta))$ and $\gamma(\theta)\neq \gamma(-\theta)$ if $\theta\neq 0$, which shows that the map $P_c(V_L)$ gives rise to a two-to-one correspondence between the line $\gamma(\theta)$ and the half-line \eqref{eqn:dominating solution:2} for $\theta\neq 0$.
The line $\gamma(\theta)$ crosses the boundary of $\mc D$ since $\gamma(1)\in \mc D$ and $\gamma(-1)\not\in \mc D$.
Let $\hat \theta$ be such that $\gamma(\hat \theta)\in \partial \mc D$.
Theorem~\ref{theorem:one-to-one boundary} implies that there does not exists $\hat V_L\neq \gamma(\hat \theta)$ such that $P_c(\gamma(\hat \theta)) = P_c(\hat V_L)$. Hence, due to the two-to-one correspondence between $\gamma(\theta)$ and \eqref{eqn:dominating solution:2} for $\theta\neq 0$, we conclude that $\hat \theta = 0$ and $\gamma(0) = v \in \partial \mc D$. Theorem~\ref{lemma:D m-matrix} implies that $-\pdd {P_c}{V_L}(v)$ is a singular M-matrix. Note that $\mu$ lies in the kernel of $-\pdd {P_c}{V_L}(v)$ due to \eqref{eqn:dominating solution:3}, and it follows from Lemma~\PIref{proposition:perron root M-matrix} that $\pm \mu>\mb 0$ and that $\mu$ spans the kernel of $-\pdd {P_c}{V_L}(v)$. 
Since $\gamma(\theta)$ intersects $\partial \mc D$ only when $\theta = 0$, and since $\gamma(1)\in\mc D$ and $\gamma(-1)\not\in\cl{\mc D}$, it follows that $\gamma(\theta)\in\mc D$ if and only if $\theta>0$.
However, if $\mu<\mb 0$ then $\gamma(\theta) = v + \theta \mu$ is a negative vector for sufficiently large $\theta$, which contradicts that all vectors in $\cl{\mc D}$ are %
positive. We conclude that $\mu>\mb 0$, which by definition of $\mu$ implies that $\tilde V_L > \tilde V_L^\prime$. The operating point $\tilde V_L$ is a strict high-voltage solution by Definition~\ref{definition:high-voltage solution}.
\end{proof}

We conclude by proving that the different types of operating points defined in Section~\ref{section:summary of part 1} 
describe one and the same operating point.

\begin{theorem}[\mainresult{desirable operating point}]\label{theorem:desirable operating point}
Let $\tilde P_c$ be a feasible vector of power demands, and let $\tilde V_L$ be an associated operating point (\ie, $\tilde P_c = P_c(\tilde V_L)$). The following statements are equivalent:
\begin{enumerate}[i)]
\item $\tilde V_L$ is long-term voltage semi-stable (\ie, $\tilde V_L\in\cl{\mc D}$);
\item $\tilde V_L$ is the unique long-term voltage semi-stable operating point associated to $\tilde P_c$;
\item $\tilde V_L$ is dissipation-minimizing;
\item $\tilde V_L$ is the unique dissipation-minimizing operating point associated to $\tilde P_c$;
\item $\tilde V_L$ is a high-voltage solution;
\item $\tilde V_L$ is a strict high-voltage solution.
\end{enumerate}
\end{theorem}
\begin{proof}
Theorem~\ref{theorem:one-to-one correspondence} guarantees the existence and uniqueness of a long-term voltage semi-stable operating point $\hat V_L$ associated to $\tilde P_c$.
It therefore suffices to show that $\hat V_L$ is the unique operating point which satisfies statements iii)-vi) individually.
Theorem~\ref{theorem:dominating solution} implies that $\hat V_L$ is a (strict) high-voltage solution. Note that there exists at most one high-voltage solution, since $\tilde V_L \le \tilde V_L^\prime$ and $\tilde V_L^\prime \le \tilde V_L$ imply that $\tilde V_L = \tilde V_L^\prime$.
Corollary~\ref{corollary:high-voltage solution point dissipation-minimizing} implies that $\hat V_L$ is the unique dissipation-minimizing operating point.
\end{proof}

Theorem~\ref{theorem:desirable operating point} shows that the desirable operating points defined in Section~\ref{section:summary of part 1} coincide, and that we may speak of a single desired operating point.
\begin{definition}\label{definition:desired operating point}
An operating point $\tilde V_L$ associated to $\tilde P_c$ is the \emph{desired operating point} if $\tilde V_L$ satisfies one (and therefore all) of the equivalent statements in Theorem~\ref{theorem:desirable operating point}.
\end{definition}
\begin{remark}
In \cite{matveev2020tool} it was shown that for a feasible vector of power demands the algorithm proposed in \cite{matveev2020tool} converges to a high-voltage solution. By Theorem~\ref{theorem:desirable operating point} this means that this algorithm converges to the unique long-term voltage semi-stable operating point associated to these power demands.
\end{remark}

\section{Conclusion}\label{section:conclusion}

In this paper we constructed a framework for the analysis of the feasibility of the power flow equations for DC power grids. Within this framework we unified and generalized the results in the literature concerning this feasibility problem, and gave a complete characterization of feasibility.

In Part~II of this paper we %
gave a necessary and sufficient condition for the feasibility (under small perturbation) for nonnegative power demands. This condition is cheaper to compute than the LMI condition proposed in Part~I.
We have presented two novel sufficient conditions for the feasibility of a power demand, which were shown to generalize known sufficient conditions in the literature.
In addition we proved that any power demand dominated by a feasible power demand is also feasible.
Finally, we showed that the operating points corresponding to a power demand which are long-term voltage semi-stable, dissipation-minimizing, or a (strict) high-voltage solution, are one and the same.

Further directions of research may concern the question if and how the approach and/or results in this paper generalize to general AC power grids.
Other interesting directions of research concern the feasibility of the power flow equations with uncertain parameters, conditions for long-term voltage (semi-)stability of an operating point, and the (non)convexity of the set of such operating points. Furthermore, control schemes which implement the proposed conditions for power flow feasibility are of particular interest.

\appendix
\renewcommand{\thelemma}{A.\arabic{theorem}}
\subsection{Properties of $\mc M$}\label{appendix:properties of M}
\begin{lemma}\label{lemma:positivity of M}
If $\mu\in\mc M$, then $\mu>\mb 0$. 
Moreover, $\mc M$ is an open cone and is simply connected.
\end{lemma}
\begin{proof}
If $\mu\in\mc M$, then $g(\mu)=[\mu]Y_{LL}+[Y_{LL}\mu]$ is a nonsingular M-matrix, and therefore a Z-matrix. Recall that $Y_{LL}$ is an irreducible Z-matrix, which implies that $(Y_{LL})_{[i,i\comp]}\lneqq 0$ for all $i$. 
If $\mu_i<0$, then $g(\mu)_{[i,i\comp]} = \mu_i (Y_{LL})_{[i,i\comp]}\gneqq 0$, which contradicts the fact that $g(\mu)$ is a Z-matrix. Hence $\mu\ge \mb 0$. %
We will show that a vector $\mu$ which contains zeros does not yield an M-matrix.
Suppose $\mu\in \mc M$ such that $\mu_{[\alpha]}=\mb 0$ and $\mu_{[\alpha\comp]} > \mb 0$ for some nonempty set $\alpha\subset\boldsymbol n$.
Since $\mu_{[\alpha]}=\mb 0$, the following submatrices of $[\mu]Y_{LL}+[Y_{LL}\mu]$ satisfy
\begin{align}\label{eqn:positivity of M:off-diagonal block}
([\mu]Y_{LL}+[Y_{LL}\mu])_{[\alpha,\alpha\comp]} = [\mu_{[\alpha]}](Y_{LL})_{[\alpha,\alpha\comp]} = 0
\end{align}
and 
\begin{align}\label{eqn:positivity of M:diagonal block}
\begin{aligned}
&([\mu]Y_{LL}+[Y_{LL}\mu])_{[\alpha,\alpha]} \\
&= [\mu_{[\alpha]}](Y_{LL})_{[\alpha,\alpha]} + [(Y_{LL}\mu)_{[\alpha]}] \\
&= 0 + \l[(Y_{LL})_{[\alpha,\alpha]} \mu_{[\alpha]} + (Y_{LL})_{[\alpha,\alpha\comp]} \mu_{[\alpha\comp]}\r]\\
&=[(Y_{LL})_{[\alpha,\alpha\comp]} \mu_{[\alpha\comp]}].
\end{aligned}
\end{align}
It follows from \eqref{eqn:positivity of M:off-diagonal block} that $[\mu]Y_{LL}+[Y_{LL}\mu]$ is block-triangular. Hence, the eigenvalues of each diagonal block are also eigenvalues of $[\mu]Y_{LL}+[Y_{LL}\mu]$. The principal submatrix given in \eqref{eqn:positivity of M:diagonal block} is such a diagonal block. Note from \eqref{eqn:positivity of M:diagonal block} that this block is diagonal, and so its eigenvalues are the elements of the vector $(Y_{LL})_{[\alpha,\alpha\comp]} \mu_{[\alpha\comp]}$. Since $Y_{LL}$ is an  irreducible Z-matrix we have $(Y_{LL})_{[\alpha,\alpha\comp]}\lneqq 0$. Recall that $\mu_{[\alpha\comp]} > \mb 0$, which implies that
\begin{align*}
(Y_{LL})_{[\alpha,\alpha\comp]} \mu_{[\alpha\comp]} \lneqq \mb 0,
\end{align*}
and so $[\mu]Y_{LL}+[Y_{LL}\mu]$ has nonpositive eigenvalues. However, since $[\mu]Y_{LL}+[Y_{LL}\mu]$ is an M-matrix, its Perron root is positive and is a lower bound for all other eigenvalues, which is a contradiction. We conclude that $\mu>\mb 0$.

The matrix $g(\mu)$ is linear in $\mu$. Hence, scaling of $\mu$ gives rise to a scaling of the eigenvalues of $g(\mu)$, and in particular of the Perron root of $g(\mu)$. Hence $\mc M$ is a cone. The set of nonsingular M-matrices is open, and so $\mc M$ is an open set.

The set $\partial \mc D$ is simply connected by Theorem~\ref{theorem:parametrization of D}. Theorem~\ref{theorem:alternative parametrization of D} shows that there exists a bicontinuous map between $\partial \mc D$ and $\mc M_1$. Topological properties are preserved by bicontinuous maps, and hence $\mc M_1$ is also simply connected. Its conic hull $\mc M$ is therefore also simply connected.
\end{proof}

\begin{lemma}\label{lemma:D subset of M}
The set of long-term voltage semi-stable operating points is contained in $\mc M$ (\ie, $\cl{\mc D} \subset \mc M$).
\end{lemma}
\begin{proof}
Recall from Theorem~\ref{lemma:D m-matrix} that if $\tilde V_L\in \cl{\mc D}$, then $-\pdd {P_c}{V_L}(\tilde V_L)$ is an M-matrix. This means that $g(\tilde V_L) - [\mc I_L^*]$ is an M-matrix by \eqref{eqn:rewrite of jacobian of P_c}. 
By adding $[\mc I_L^*]$ to $g(\tilde V_L) - [\mc I_L^*]$, Proposition~\ref{P1-proposition:diagonal properties}:\PIref{proposition:diagonal properties:irreducible M-matrix addition} implies that $g(\tilde V_L)$ is an M-matrix since $\mc I_L^*\gneqq \mb 0$. 
\end{proof}
\subsection{Proof of Theorem~\ref{theorem:generalization simpson-porco condition}}\label{appendix:proof of sufficient condition}\noindent

For the sake of notation we follow Lemma~\ref{lemma:single source} and define $\hat Y_{LL}:=[V_L^*]Y_{LL}[V_L^*]$, which is an irreducible nonsingular M-matrix. It follows from \cite[Thm. 5.12]{fiedler1986special} that the inverse of $\hat Y_{LL}$ is positive. %
Let $S$ be the set of $\tilde P_c$ defined by 
\begin{align}\label{eqn:generalization simpson-porco condition proof:1}
\tilde P_c \ge \mb 0;\quad  \hat Y_{LL}\mathstrut \inv \tilde P_c \le \tfrac 1 4 \mb 1,
\end{align}
which corresponds to all $\tilde P_c\in\mc N$ so that \eqref{eqn:generalization simpson-porco condition} holds.
The set $S$ is convex and \eqref{eqn:generalization simpson-porco condition proof:1} describes the intersection of $2n$ closed half-spaces. The normals to these half-spaces are given by the canonical basis vectors $e_1,\dots,e_n$ and the rows of $\hat Y_{LL}\mathstrut\inv$.
The set $S$ is bounded since $\hat Y_{LL}\mathstrut \inv$ is positive.
Weyl's Theorem \cite[pp. 88]{magaril2003convex} states that $S$ is the convex hull of the points which lie on the boundary of $n$ half-spaces in \eqref{eqn:generalization simpson-porco condition proof:1} so that their corresponding normals span $\RR^n$.
To this end we define $P_c^\emptyset :=\mb 0$, which lies on the boundary of the $n$ half-spaces described by $\tilde P_c \ge \mb 0$.
Similarly, we let $\alpha\subset \boldsymbol n$ be nonempty and let $P_c^\alpha\in S$ be a point described by Weyl's Theorem for which $(\hat Y_{LL}\mathstrut \inv \tilde P_c)_{[\alpha]} = \tfrac 1 4 \mb 1$ and $(\hat Y_{LL}\mathstrut \inv \tilde P_c)_{[\alpha\comp]} < \tfrac 1 4 \mb 1$. The corresponding normals are given by the rows of $\hat Y_{LL}\mathstrut \inv$ indexed by $\alpha$. Since $\hat Y_{LL}$ is positive definite we know that $(\hat Y_{LL}\mathstrut)_{[\alpha,\alpha]}$ is positive definite and therefore nonsingular. 
The only choice of normals of the half-spaces which complete the span of $\RR^n$ are $e_i$ for $i\in\alpha\comp$, which implies $(P_c^\alpha)_{[\alpha\comp]}=\mb 0$. Since $(P_c^\alpha)_{[\alpha\comp]}=\mb 0$ we have 
\begin{align*}
\tfrac 1 4 \mb 1 = (\hat Y_{LL}\mathstrut \inv P_c^\alpha)_{[\alpha]}  = (\hat Y_{LL}\mathstrut \inv)_{[\alpha,\alpha]} (P_c^\alpha)_{[\alpha]},
\end{align*}
and therefore $(P_c^\alpha)_{[\alpha]} = \tfrac 1 4 (\hat Y_{LL}\mathstrut \inv)_{[\alpha,\alpha]}\mathstrut\inv\mb 1$.
By the block matrix inverse formula \cite[Eq. (0.8.1)]{zhang2006schur} we observe that
\begin{align}\label{eqn:generalization simpson-porco condition proof:5}
(P_c^\alpha)_{[\alpha]} = \tfrac 1 4 (\hat Y_{LL} / (\hat Y_{LL})_{[\alpha\comp,\alpha\comp]}) \mb 1.
\end{align}
The above exhaustively describes all points specified by Weyl's Theorem, and hence we have
\begin{align*}
S = \conv{\set{P_c^\alpha}{\alpha\subset\boldsymbol n}}.
\end{align*}
Recall that $P_c^\emptyset = \mb 0 = P_c(V_L^*)\in\inter{\mc F}$. 
The points $P_c^\alpha$ for nonempty $\alpha\subset\boldsymbol n$ correspond to the power demands described in \eqref{eqn:generalization simpson-porco condition powers} through substitution of $\hat Y_{LL}=[V_L^*]Y_{LL}[V_L^*]$. %
We show that these points lie on the boundary of $\mc F$.
Note that for $\alpha=\boldsymbol n$ we have by \eqref{eqn:generalization simpson-porco condition proof:5}, \eqref{eqn:open-circuit voltages} and \eqref{eqn:max total power demand} that
\begin{align*}
P_c^{\boldsymbol n} = \tfrac 1 4 \hat Y_{LL}\mb 1 = \tfrac 1 4[V_L^*]Y_{LL}[V_L^*]\mb 1 = \tfrac 1 4[V_L^*]\mc I_L^* = P_{\text{max}},
\end{align*}
which lies on the boundary of $\mc F$.
Consider any feasible power demand $\tilde P_c\in\mc F$ such that $(\tilde P_c)_{[\alpha]}=\mb 0$ with  $\alpha\neq\emptyset,\boldsymbol n$. 
Let $\tilde V_L>\mb 0$ be a so that $\tilde P_c = P_c(\tilde V_L)$.
By \eqref{eqn:definition of f} we have 
\begin{align}\label{eqn:generalization simpson-porco condition proof:3}
\mb 0 = P_c(\tilde V_L)_{[\alpha\comp]} = [(\tilde V_L)_{[\alpha\comp]}] (Y_{LL}(V_L^* - \tilde V_L))_{[\alpha\comp]}
\end{align} 
where we used \eqref{eqn:open-circuit voltages}. Since $(\tilde V_L^\alpha)_{[\alpha\comp]} > \mb 0$ it follows from \eqref{eqn:generalization simpson-porco condition proof:3} that $(Y_{LL}(V_L^* - \tilde V_L))_{[\alpha\comp]} = \mb 0$. 
Since $(Y_{LL})_{[\alpha\comp,\alpha\comp]}$ is nonsingular, we may solve for $(\tilde V_L)_{[\alpha\comp]}$. Similar to \cite[Eq. (0.7.4)]{zhang2006schur}, substitution of $(\tilde V_L)_{[\alpha\comp]}$ in $P_c(\tilde V_L)_{[\alpha\comp]}$ yields
\begin{align*}%
P_c(\tilde V_L)_{[\alpha]} = [(\tilde V_L)_{[\alpha]}]Y_{LL}/(Y_{LL})_{[\alpha\comp,\alpha\comp]}((V_L^*)_{[\alpha]} - (\tilde V_L)_{[\alpha]}),
\end{align*}
which corresponds to the power flow equations of a Kron-reduced power grid (see, \eg, \cite{dorfler2012kron,schaft2019flow}). Analogous to Lemma~\ref{lemma:maximizing power demand}, the maximizing feasible power demand for the Kron-reduced power grid is obtained by taking $(\tilde V_L)_{[\alpha]} = \tfrac 1 2 (V_L^*)_{[\alpha]}$, which corresponds in the power demand $P_c^{\alpha}$. 
Hence $P_c^{\alpha}$ lies on the boundary of $\mc F$.
Since $\mc F$ is convex by Theorem~\ref{theorem:convexity of F}, and $P_c^{\alpha}\in\mc F$ for all $\alpha\subset\boldsymbol n$, we have that $S\subset \mc F\cap \mc N$. 
Each supporting half-space of $\mc F$ has a unique point of support (Theorem~\PIref{theorem:supporting half-spaces of F}), and so the boundary of $\mc F$ does not contain a line piece. Consequently, the all points in $S$ other than the points $P_c^\alpha$ for $\alpha\neq\emptyset$ lie in the interior of $\mc F$. Lemma~\ref{lemma:power demand domination} implies \eqref{eqn:generalization simpson-porco condition} from \eqref{eqn:generalization simpson-porco condition proof:1}.\hfill\QED


\begin{thebibliography}{10}
\providecommand{\url}[1]{#1}
\csname url@samestyle\endcsname
\providecommand{\newblock}{\relax}
\providecommand{\bibinfo}[2]{#2}
\providecommand{\BIBentrySTDinterwordspacing}{\spaceskip=0pt\relax}
\providecommand{\BIBentryALTinterwordstretchfactor}{4}
\providecommand{\BIBentryALTinterwordspacing}{\spaceskip=\fontdimen2\font plus
\BIBentryALTinterwordstretchfactor\fontdimen3\font minus
  \fontdimen4\font\relax}
\providecommand{\BIBforeignlanguage}[2]{{%
\expandafter\ifx\csname l@#1\endcsname\relax
\typeout{** WARNING: IEEEtran.bst: No hyphenation pattern has been}%
\typeout{** loaded for the language `#1'. Using the pattern for}%
\typeout{** the default language instead.}%
\else
\language=\csname l@#1\endcsname
\fi
#2}}
\providecommand{\BIBdecl}{\relax}
\BIBdecl

\bibitem{lof1993analysis}
P.-A. L\"of, D.~J. Hill, S.~Arnborg, and G.~Andersson, ``On the analysis of
  long-term voltage stability,'' \emph{International Journal of Electrical
  Power \& Energy Systems}, vol.~15, no.~4, pp. 229 -- 237, 1993.

\bibitem{62415}
D.~J. {Hill} and I.~M.~Y. {Mareels}, ``Stability theory for
  differential/algebraic systems with application to power systems,''
  \emph{IEEE Transactions on Circuits and Systems}, vol.~37, no.~11, pp.
  1416--1423, Nov 1990.

\bibitem{tinney1967power}
W.~F. {Tinney} and C.~E. {Hart}, ``Power flow solution by newton's method,''
  \emph{IEEE Transactions on Power Apparatus and Systems}, vol. PAS-86, no.~11,
  pp. 1449--1460, 1967.

\bibitem{dorfler2013novel}
F.~{D\"orfler} and F.~{Bullo}, ``Novel insights into lossless ac and dc power
  flow,'' in \emph{2013 IEEE Power Energy Society General Meeting}, 2013, pp.
  1--5.

\bibitem{dymarsky2014convexity}
A.~Dymarsky, ``On the convexity of image of a multidimensional quadratic map,''
  \emph{arXiv preprint arXiv:1410.2254}, 2014.

\bibitem{bolognani2015existence}
S.~Bolognani and S.~Zampieri, ``On the existence and linear approximation of
  the power flow solution in power distribution networks,'' \emph{IEEE
  Transactions on Power Systems}, vol.~31, no.~1, pp. 163--172, 2015.

\bibitem{simpson2016voltage}
J.~W. Simpson-Porco, F.~D{\"o}rfler, and F.~Bullo, ``Voltage collapse in
  complex power grids,'' \emph{Nature Communications}, vol.~7, no. 10790, 2016.

\bibitem{barabanov2016}
N.~{Barabanov}, R.~{Ortega}, R.~{Gri\~n\'o}, and B.~{Polyak}, ``On existence
  and stability of equilibria of linear time-invariant systems with constant
  power loads,'' \emph{IEEE Transactions on Circuits and Systems I: Regular
  Papers}, vol.~63, no.~1, pp. 114--121, Jan 2016.

\bibitem{matveev2020tool}
A.~S. {Matveev}, J.~E. {Machado}, R.~{Ortega}, J.~{Schiffer}, and A.~{Pyrkin},
  ``A tool for analysis of existence of equilibria and voltage stability in
  power systems with constant power loads,'' \emph{IEEE Transactions on
  Automatic Control}, vol.~65, no.~11, pp. 4726--4740, 2020.

\bibitem{fiedler1986special}
M.~Fiedler, \emph{Special matrices and their applications in numerical
  mathematics}.\hskip 1em plus 0.5em minus 0.4em\relax Kluwer Academic
  Publishers, 1986.

\bibitem{schaft2010characterization}
A.~van~der Schaft, ``Characterization and partial synthesis of the behavior of
  resistive circuits at their terminals,'' \emph{Systems \& Control Letters},
  vol.~59, no.~7, pp. 423 -- 428, 2010.

\bibitem{schaft2019flow}
A.~van~der {S}chaft, ``The flow equations of resistive electrical networks,''
  in \emph{Interpolation and Realization Theory with Applications to Control
  Theory: In Honor of Joe Ball}.\hskip 1em plus 0.5em minus 0.4em\relax
  Springer International Publishing, 2019, pp. 329--341.

\bibitem{rudin1964principles}
W.~Rudin \emph{et~al.}, \emph{Principles of mathematical analysis}.\hskip 1em
  plus 0.5em minus 0.4em\relax McGraw-Hill New York, 1964, vol.~3.

\bibitem{dorfler2012kron}
F.~D{\"o}rfler and F.~Bullo, ``Kron reduction of graphs with applications to
  electrical networks,'' \emph{IEEE Transactions on Circuits and Systems I:
  Regular Papers}, vol.~60, no.~1, pp. 150--163, 2012.

\bibitem{roman2008advanced}
S.~Roman, \emph{Advanced linear algebra}.\hskip 1em plus 0.5em minus
  0.4em\relax Springer, 2008, vol.~3.

\bibitem{magaril2003convex}
G.~G. Magaril-Il'yaev and V.~M. Tikhomirov, \emph{Convex analysis: theory and
  applications}.\hskip 1em plus 0.5em minus 0.4em\relax American Mathematical
  Soc., 2003, vol. 222.

\bibitem{zhang2006schur}
F.~Zhang, \emph{The Schur complement and its applications}.\hskip 1em plus
  0.5em minus 0.4em\relax Springer Science \& Business Media, 2006, vol.~4.

\end{thebibliography}
\end{document}